\documentclass{amsart}
\usepackage{amsmath,amssymb,amsfonts,bm}
\usepackage{mathrsfs,latexsym,amsthm,enumerate}
\usepackage{tikz}
\usepackage{amscd}
\usepackage[all]{xy}
\usepackage{hyperref}

\numberwithin{equation}{section}

\makeatletter
\@namedef{subjclassname@2020}{%
  \textup{2020} Mathematics Subject Classification}
\makeatother

\newtheorem{maintheorem}{Theorem}
\newtheorem{lemma}{Lemma}[section]
\newtheorem{proposition}[lemma]{Proposition}
\newtheorem{corollary}[lemma]{Corollary}
\newtheorem{defn}[lemma]{Definition}
\newtheorem*{defnx}{Definition}
\newtheorem{rem}[lemma]{Remark}
\newtheorem{theorem}[lemma]{Theorem}

\newtheorem{exm}[lemma]{Example}

\newenvironment{definition*}{\begin{defnx}\em}{\end{defnx}} 
\newenvironment{example}{\begin{exm}\em}{\end{exm}}
\newenvironment{remark}{\begin{rem}\em}{\end{rem}}

\newcommand\sections{\mathit\Gamma}
\newcommand\psections{\sections^{\mathrm{pr}}}
\newcommand\bisections{\sections^{\mathrm{bi}}}
\newcommand\spp{\varsigma}
\newcommand\tspp{\tau}
\newcommand\lcc{\operatorname{{\mathcal L}^{\vee}}}
\newcommand\lc{\operatorname{{\mathcal L}}}
\newcommand\V{\bigvee}
\newcommand\cf{cf.}
\newcommand\ipi{{\operatorname{\mathcal I}}}
\newcommand\st{\colon}
\newcommand\opens{\operatorname{\mathcal O}}
\newcommand\SL{\mathsf{SL}}

\newcommand\Sh{\mathsf{Sh}}
\newcommand\Mod{\mathsf{Mod}}
\newcommand\then{\mathop{\&}}
\newcommand\ident{\operatorname{id}}

\begin{document}

\title{Morita equivalence of pseudogroups}

\author{M. V. Lawson}
\address{Mark V. Lawson, 
Department of Mathematics
and the
Maxwell Institute for Mathematical Sciences, 
Heriot-Watt University,
Riccarton,
Edinburgh~EH14~4AS, UNITED KINGDOM}
\email{m.v.lawson@hw.ac.uk}

\author{P. Resende}
\address{Pedro Resende, 
Centre for Mathematical Analysis, Geometry, and Dynamical Systems,
Departamento de Matem\'atica,
Instituto Superior T\'ecnico,
Universidade de Lisboa, 
Av. Rovisco Pais,
1049-001 Lisboa, PORTUGAL}
\email{pedro.m.a.resende@tecnico.ulisboa.pt}

\thanks{The authors would like to thank the anonymous referee for their sterling work in reading the first draft of this paper.
The second author's research was funded by FCT/Portugal through project UIDB/04459/2020.}

\begin{abstract} We take advantage of the correspondence between pseudogroups and inverse quantal frames, and of the recent description of Morita equivalence for inverse quantal frames
in terms of biprincipal bisheaves, to define Morita equivalence for pseudogroups and to investigate its applications.
In particular, two pseudogroups are Morita equivalent if and only if their corresponding localic \'etale groupoids are.
We explore the clear analogies between our definition of Morita equivalence for pseudogroups and the
usual notion of strong Morita equivalence for $C^{\ast}$-algebras and these lead to a number of concrete results. 
\end{abstract}

\keywords{Inverse semigroups,  pseudogroups,  Morita equivalence, quantales, inverse quantal frames, \'etale groupoids}

\subjclass[2020]{18F75, 20M18, 22A22}

\maketitle

\tableofcontents

\section{Introduction}

A pseudogroup is a complete inverse semigroup in which multiplication distributes over compatible joins.
This is an abstract structure, whereas pseudogroups of transformations are concrete exemplars:
for example, the pseudogroup of all partial homeomorphisms between the open subsets of a topological space.
The path to the development of an {\em abstract} theory of pseudogroups has been a long and circuitous one;
see \cite{Lawson1998}, for further background.

The goal of this paper is to describe the Morita theory of pseudogroups.
Notice that there is a pre-existing theory of Morita equivalence for inverse semigroups \cite{Steinberg, FLS};
the theory we shall develop for pseudogroups is similar to but different from this one. In fact, both theories go
back to the Morita theory described by Rieffel for C*-algebras in  \cite{Rieffel1974, Rieffel1982},
with ours being a much closer analogue of the latter. 
Concretely, we shall rely on the fact that there are  two distinct philosophies in handling pseudogroups:
they can be regarded as a class of inverse semigroups, as we have described above,
or they can be connected to quantales --- specifically, inverse quantal frames  \cite{Resende2007}.
These two philosophies are related via a dictionary \cite[Th.~1.6]{Resende2007} in the form of
two functors that define an equivalence of categories: from the pseudogroup $S$ one can construct the inverse quantal frame
$\lcc(S)$, and from the inverse quantal frame $Q$ one can construct the pseudogroup
of partial units $Q_\ipi$. The goal of this paper is to describe the Morita
theory of pseudogroups within the inverse semigroup philosophy, for which we shall take advantage of the sheaf theory~\cite{Resende2012} of inverse quantal frames and their recently developed \cite{QR} 
Morita theory based on Hilsum--Skandalis maps.
We prove three main theorems, the first of which translates between the two philosophies:

\begin{maintheorem}\label{theoremA}
Let $S$ and $T$ be pseudgroups. The following conditions are equivalent:
\begin{enumerate}
\item $S$ and $T$ are Morita equivalent pseudogroups.
\item $\lcc(S)$ and $\lcc(T)$ are Morita equivalent inverse quantal frames.
\item $\mathsf{G}(S)$ and $\mathsf{G}(T)$ are Morita equivalent localic \'etale groupoids.
\end{enumerate}
\end{maintheorem}

The second is analogous to a similar result in ring theory:

\begin{maintheorem}\label{theoremB}
Let $S$ and $T$ be pseudogroups.
Then $S$ and $T$ are Morita equivalent if and only if there is an $(S,T)$-equivalence bimodule.
\end{maintheorem}

The third is in the spirit of the characterization of Morita equivalence in inverse semigroup theory \cite{FLS}:

\begin{maintheorem}\label{theoremC}
Let $S$ and $T$ be pseudogroups.
Then there is an $(S,T)$-equivalence bimodule
if and only if  $S$ and $T$ have a joint sup-enlargement.
\end{maintheorem}

\begin{remark}
In this paper, $\mathsf{G}(S)$ will denote the {\em localic} \'etale groupoid associated with the pseudogroup $S$.
In the case where $S$ is spatial, this is the same as the groupoid of completely prime filters studied in \cite{LL}.
\end{remark}

We shall explore the theory begun here for the case of coherent pseudogroups in a follow-up paper.

\section{Background}

For standard results on inverse semigroups, we refer the reader to \cite{Lawson1998}.
We recall some important definitions here.
If $S$ is any semigroup and $I$ is a subset, we say that $I$ is a {\em (semigroup) ideal} if $SI, IS \subseteq S$.
The set of idempotents of $S$ is denoted by $\mathsf{E}(S)$.
Recall that any two idempotents in an inverse semigroup commute;
we shall use this fact without comment many times in the sequel.
In an inverse semigroup $S$, the idempotent $s^{-1}s$ is denoted by $\mathbf{d}(s)$ and the idempotent $ss^{-1}$ by $\mathbf{r}(s)$.
We write $e \stackrel{a}{\rightarrow} f$ if $e = \mathbf{d}(a)$ and $f = \mathbf{r}(a)$;
in this case, we also write $e \, \mathscr{D} \, f$.
A product of the form $ab$ where $\mathbf{d}(a) = \mathbf{r}(b)$ is called a {\em reduced product}.
If $ab$ is any product then $ab = a'b'$ where $a'b'$ is a reduced product;
simply put $a' = a\mathbf{r}(b)$ and $b' = \mathbf{d}(a)b$.
The natural partial order on an inverse semigroup is denoted by $\leq$.
We define $s \sim_{l} t$ if $s^{-1}t$ is an idempotent and say that $s$ and $t$ are {\em left compatible};
this is equivalent to $\mathbf{r}(s)t = \mathbf{r}(t)s$.
We define $s \sim_{r} t$ if $st^{-1}$ is an idempotent and say that $s$ and $t$ are {\em right compatible};
this is equivalent to $t\mathbf{d}(s) = s\mathbf{d}(t)$.\footnote{Beware! We have changed parity compared to the definitions in \cite{Lawson1998}.}
We write $s \sim t$ and say they are {\em compatible} if they are both left and right compatible.
A subset is compatible if it is empty or, being non-empty, each pair of elements in it is compatible.
We now come to the main class of semigroups studied in this paper.

\begin{definition*}
An inverse semigroup $S$ is called a {\em pseudogroup} if every compatible subset (recall that the join of the empty set is $0$) 
has a join and multiplication distributes over all such joins.
This means that if $A$ is any compatible subset of $S$ then $\bigvee_{a\in A} a$ exists and for any $s \in S$ we have that
$$s \bigl( \bigvee_{a\in A} a  \bigr) = \bigvee_{a\in A} sa
\quad\text{ and }\quad
\bigl( \bigvee_{a\in A} a  \bigr)s = \bigvee_{a\in A} as.$$
You can easily check that if $a \sim b$ then $sa \sim sb$ and
$as \sim bs$.
Thus the joins on the right-hand sides of the above two equations are actually defined.
\end{definition*}

If $S$ is a pseudogroup then $\mathsf{E}(S)$ is a frame;
that is, a complete infinitely distributive lattice.
For more information on frames, see \cite{J,PP}.
We say that $S$ is {\em spatial} if $\mathsf{E}(S)$ is spatial.
Observe that pseudogroups automatically have all binary meets \cite{Resende2006}.
The identity of the pseudogroup $S$ will be denoted by $e_{S}$.
It is the top element of $\mathsf{E}(S)$.
We use the word {\em sup} and the notation $\bigvee$ for arbitrary joins.
If $X \subseteq S$, where $S$ is a pseudogroup, 
then we denote by $X^{\V}$ the set of all joins of compatible subsets of $X$.

Let $(P,\leq)$ be a partially ordered set or, simply, poset.
If $X \subseteq P$ define
$$X^{\uparrow} = \{y \in P \st x \leq y \text{ for some } x \in X\}$$
and
$$X^{\downarrow} = \{y \in P \st y \leq x \text{ for some } x \in X\}.$$
If $X = \{x\}$ we write $x^{\uparrow}$ and $x^{\downarrow}$ instead of $\{x\}^{\uparrow}$ and $\{x\}^{\downarrow}$, respectively.
If $X = X^{\downarrow}$ we say that $X$ is an {\em order-ideal} or {\em downward closed}.
An order-preserving map between posets is said to be {\em monotone}.

\section{Pseudogroup modules}\label{sec:pseudogroupmodules}

Often, two structures $\mathcal{A}$ and $\mathcal{B}$ are said to be Morita equivalent if the category of left modules of $\mathcal{A}$
is equivalent to the category of left modules of $\mathcal{B}$. A good account of the Morita theory of rings can be found in \cite{AF}.
The goal of this section is to define what we mean by a left module of a pseudogroup in order that we may define when two pseudogroups are Morita equivalent.

\subsection{Supported actions}

The definitions of semigroup action and monoid action are well-known, and here we shall work with left actions.
The definitions for right actions are made in the obvious way.

Let $S$ be an inverse semigroup.
Recall that a left {\em $S$-action} is defined by a semigroup action $S \times X \rightarrow X$ on a set $X$.
We shall need the concept of a `supported' left action of $S$.\footnote{This is exactly what Steinberg calls an \'etale action \cite{Steinberg}.}
A left $S$-action on a set $X$ is said to be {\em supported}, with support $p \colon X \rightarrow \mathsf{E}(S)$,
if the following two conditions hold:
\begin{enumerate}

\item $p(x)x = x$ for all $x \in X$.

\item $p(sx) = sp(x)s^{-1}$ for all $x \in X$ and $s \in S$.

\end{enumerate}
Note that if $s\in E(S)$ then we have $p(sx)=sp(x)$, so $p$ is $E(S)$-equivariant.

Let $(X,p)$ and $(Y,q)$ be two supported left $S$-actions.
A {\em homomorphism} $\theta$ from $(X,p)$ to $(Y,q)$ is a function
$\theta \colon X \rightarrow Y$ that satisfies two conditions:
\begin{enumerate} 

\item $\theta (sx) = s \theta (x)$ for all $s \in S$ and $x \in X$.

\item $q(\theta (x)) = p(x)$.

\end{enumerate}
We will use standard results about supported actions proved in \cite{Steinberg}.
In particular, we may define an order on $X$ by $x \leq y$ if and only if $x = ey$ for some $e \in \mathsf{E}(S)$.
It is routine to check that, in fact, $x \leq y$ if and only if $x = p(x)y$, and that if $S$ is a monoid the action is necessarily unital.
This order is compatible with the action so that if $x \leq y$ then $sx \leq sy$ and $p \colon X \rightarrow \mathsf{E}(S)$ 
is order-preserving \cite[Prop.~3.2]{Steinberg}.

Let $(X,p)$ be a supported action of the pseudogroup $S$.
In what follows, notation originally defined for pseudogroups is also defined for their actions.
We say that $x, y \in X$ are {\em left compatible}, denoted by $x \sim_{l} y$, if $p(x)y = p(y)x$.
Observe that if $x$ and $y$ have an upper bound they are left compatible \cite[Lemma~3.4]{Steinberg}.

\begin{remark}
If we work with right actions then we have the analogous definition of `right compatible.'
\end{remark}

\begin{lemma}\label{lem:day} Let $X$ be a supported left $S$-action.
\begin{enumerate}
\item If  $x \sim_{l} y$ then $sx \sim_{l} sy$ for any $s \in S$.
\item If $s \sim_{l} t$ then $sx \sim_{l} tx$ for any $x \in X$.
\end{enumerate}
\end{lemma}
\begin{proof} (1) We calculate $p(sy)sx$.
This equals
$$sp(y)s^{-1}sx = sp(y)x = sp(x)y = s(s^{-1}s)p(x)y = sp(x)s^{-1} sy = p(sx)sy,$$
using the fact that $p(x)y = p(y)x$.
Thus $sx$ and $sy$ are left compatible.

(2) We calculate $p(sx)tx$.
This equals $sp(x)s^{-1}tx = ss^{-1}tp(x)x = ss^{-1}tp(x)x = tt^{-1}sp(x)x = tp(x)t^{-1}sx = p(tx)sx$,
where we use the fact that since $s \sim_{l} t$ we have $ss^{-1}t = tt^{-1}s$ and $s^{-1}t$ is an idempotent.
It follows that $sx$ and $tx$ are left compatible.
\end{proof}

\subsection{Modules and Morita equivalence}\label{sec:modules}

Now, let $S$ be a pseudogroup.
We say that $X$ is a {\em left $S$-module} if the following conditions hold:
\begin{enumerate}
\item There is a supported left action of $S$ on the set $X$ with support $p \colon X \rightarrow \mathsf{E}(S)$.
\item Every left compatible subset of $X$ has a join; in particular, the set $X$ has a smallest element, also denoted by $0$,
which is the join of the empty set.
\item The action distributes over left compatible joins. (This makes sense by part (1) of Lemma~\ref{lem:day}).
Thus $s \left( \V_{i \in I} x_{i} \right) = \V_{i \in I} sx_{i}$.
\item If $\V_{i \in I} s_{i}$ exists then $\left( \V_{i \in I} s_{i} \right)x = \V_{i \in I} s_{i}x$.
(This makes sense by part (2) of  Lemma~\ref{lem:day}).
\end{enumerate}

Let us prove some important properties of left $S$-modules.

\begin{lemma}\label{lem:groucho} Let $S$ be a pseudogroup and let $X$ be a left $S$-module.
If $\V_{i \in I} x_{i}$ exists
then $p\left(\V_{i \in I} x_{i} \right) = \V_{i \in I} p(x_{i})$.
\end{lemma}
\begin{proof} Put $x = \V_{i \in I} x_{i}$.
Then $x_{i} \leq x$ for all $i \in I$ and so $\V_{i \in I} p(x_{i}) \leq p(x)$.
Suppose that $p(x_{i}) \leq e$ for all $i \in I$ where $e \in \mathsf{E}(S)$.
Then $ex = \V_{i \in I}ex_{i} = \V_{i \in I} x_{i} = x$.
Thus $p(ex) = p(x)$ and so $p(x) \leq e$.
It follows that $p(x) = \V_{i \in I} p(x_{i})$, as claimed.
\end{proof}

\begin{lemma}\label{lem:harpo} Let $S$ be a pseudogroup and let $X$ be a left $S$-module.
The poset $X$ has all binary meets.
\end{lemma}
\begin{proof} Let $x,y \in X$.
The set $x^{\downarrow} \cap y^{\downarrow}$ contains the smallest element of $X$ and is a set of left compatible elements.
Put $z$ equal to the join of the set  $x^{\downarrow} \cap y^{\downarrow}$.
Then $z = x \wedge y$, by construction.
\end{proof}

\begin{lemma}\label{lem:zeppo} Let $S$ be a pseudogroup and let $X$ be a left $S$-module.
Suppose that $\V_{i \in I} y_{i}$ exists and that $x \in X$.
Then $\V_{i \in I} x \wedge y_{i}$ exists and 
$$x \wedge \bigl( \V_{i \in I} y_{i} \bigr) = \V_{i \in I} x \wedge y_{i}.$$
\end{lemma}
\begin{proof} We first obtain a more useful description of binary meets.
Let $u,v \in X$.
Define $f = \V \{ g \leq p(u)p(v) \st gu = gv \}$.
By construction $fu = fv \leq u,v$.
Suppose that $z \leq u,v$.
Then $z = p(z)u = p(z)v$.
Thus $p(z) \leq f$ and so $z \leq fu = fv$.
It follows that $u \wedge v = fu = fv$.

Observe that $x \wedge y_{i}$ are all bounded above by $x$ and so are left compatible.
It follows that 
$x \wedge \left( \V_{i \in I} y_{i} \right)$
and
$\V_{i \in I} x \wedge y_{i}$
both actually exist.
For each $j \in I$, we have that
$x \wedge y_{j} \leq x \wedge \left( \V_{i \in I} y_{i} \right)$.
Thus
$\V_{i \in I} x \wedge y_{i} \leq x \wedge \left( \V_{i \in I} y_{i} \right)$.
It remains to prove the reverse inequality.
Put $y = \V_{i \in I} y_{i}$.
Define $g = \V \{g' \leq p(x)p(y) \st g'x = g'y\}$.
Thus $x \wedge y = gx = gy$.
Define $g_{i} = \V \{ g_{i}' \leq p(x)p(y_{i}) \st g_{i}'x = g_{i}'y_{i} \}$.
Then $x \wedge y_{i} = g_{i}x = g_{i}y_{i}$.

It is easy to show that $p(y_{i})g \leq g_{i}$.
To show that $g_{i} \leq p(y_{i})g$ observe that $g_{i} \leq p(y_{i})$.
It follows that $g_{i}x = g_{i}y$.
We have therefore shown that $g_{i} = p(y_{i})g$.

We now calculate:
$$x \wedge y = gy = g\bigl( \V_{i \in I} y_{i} \bigr) = \V_{i \in I} gp(y_{i})y_{i} = \V_{i \in I} g_{i}y_{i} = \V_{i \in I} x \wedge y_{i}.$$ 
\end{proof}

\begin{remark}
Let $S$ be a pseudogroup and
let $X$ be a left $S$-module.
Then the action $S \times X \rightarrow X$ is in fact a {\em monoid} action.
Observe that $1 = \bigvee_{e \in \mathsf{E}(S)} e$.
Then $1x = (\bigvee_{e \in \mathsf{E}(S)}e)x = \bigvee_{e \in \mathsf{E}(S)} ex$.
But $p(x)x = x$ occurs and otherwise $ex \leq x$.
It therefore follows that $1x = x$, as required.
\end{remark}

A {\em homomorphism} of left $S$-modules is simply a homomorphism of the underlying supported left $S$-actions that preserves
all left compatible joins.
If $S$ is a pseudogroup, we denote the category of left $S$-modules and their homomorphisms by $S$-$\Mod$.
We now come to the key definition of this paper.

\begin{definition*}
Let $S$ and $T$ be pseudogroups.
We say that $S$ and $T$ are {\em Morita equivalent} 
if their categories $S$-$\Mod$ and $T$-$\Mod$ are equivalent.
\end{definition*}

\subsection{Pointed actions}

Let $S$ be an inverse semigroup with zero.
We say that a left $S$-action on a set $X$ is {\em pointed} if there
is a minimum element $0 \in X$ such that $s0= 0$ for all $s \in S$, and $0x = 0$ for all $x \in X$.

\begin{lemma}
Let $S$ be a pseudogroup. Every left $S$-module $X$ is a pointed left $S$-action.
\end{lemma}
\begin{proof} We have already observed in the definition that $X$ has a smallest element denoted by $0$.
By the distributivity condition (3) in the same definition, $s0 = 0$ for all $s \in S$.
And, by condition (4), $0x = 0$ for all $x \in X$.
It follows that left $S$-modules are always pointed.
\end{proof}

Recall, that \cite{Schein} showed that every inverse semigroup could be completed to a pseudogroup
by considering the set of all compatible order-ideals of the inverse semigroup.
The following is the analogue of Schein's result for actions. 

\begin{lemma}\label{lem:scampton} Let $S$ be a pseudogroup and let $S \times X \rightarrow X$
be a pointed supported left action with support  $p \colon X \rightarrow \mathsf{E}(S)$.
Denote by $\mathsf{L}(X)$ the set of all left compatible order-ideals of $X$.
\begin{enumerate}
\item $\mathsf{L}(X)$ can be made into a left $S$-module.
\item There is a homomorphism of $\iota \colon X \rightarrow \mathsf{L}(X)$, given by $\iota (x) = x^{\downarrow}$, of supported left $S$-actions.
\item $\iota \colon X \rightarrow \mathsf{L}(X)$ is universal.
\end{enumerate}
\end{lemma}
\begin{proof} (1) We first of all define an action $S \times \mathsf{L}(X) \rightarrow \mathsf{L}(X)$.
Let $A \subseteq X$ be a left compatible order-ideal.
We prove that $sA$ is also a left compatible order-ideal.
Let $x \leq sy$ where $y \in A$.
Then $s^{-1}x \leq s^{-1}s y \leq y$.
It follows that $s^{-1}x \in A$.
Thus $ss^{-1}x \in sA$.
But $ss^{-1}x = x$.
It follows that $sA$ is an order-ideal.
It remains to show that $sA$ is a left compatible subset.
But this follows by part (1) of Lemma~\ref{lem:day}.
We therefore have an action $S \times \mathsf{L}(X) \rightarrow \mathsf{L}(X)$.

Define $\widehat{p} \colon \mathsf{L}(X) \rightarrow \mathsf{E}(S)$ by
$\widehat{p}(A) = \V_{x \in A} p(x)$.
It can be easily checked that this really is a support.

Let $A,B \in \mathsf{L}(X)$.
Suppose that $\widehat{p}(A)B = \widehat{p}(B)A$. 
Let $x \in A$ and $y \in B$.
We prove $x$ and $y$ are left compatible.
Put $e = \widehat{p}(A) = \V_{a \in A} p(a)$ and $f = \widehat{p}(B) = \V_{b \in B} p(b)$.
By assumption, $ey = fx_{1}$ and $fx = ey_{1}$, where $x_{1} \in A$ and $y_{1} \in B$.
Observe that since $y \in B$ we have that $p(y) \leq f$.
Thus
$p(y)x = p(y)p(x)x =  p(x)p(y)fx = p(x)p(y)ey_{1}$.
Now $p(y)y_{1} = p(y_{1})y$.
Thus $p(y)x = ep(x) p(y)y_{1} = e p(x) p(y_{1})y = ep(y_{1})p(x)y \leq p(x)y$.
By symmetry, we have that $p(x)y \leq p(y)x$ and so $x \sim _{l} y$, as required.
It follows that if $\mathscr{A} = \{A_{i} \st i \in I \}$ is any left compatible subset of $\mathsf{L}(X)$,
then $\bigcup \mathscr{A} \in \mathsf{L}(X)$.
It is now routine to check that $\mathsf{L}(X)$ is a left $S$-module.

(2) This is a straightforward verification.

(3) Let $(M,q)$ be  a left $S$-module and let $\alpha \colon X \rightarrow M$ be
a homomorphism. Then there is a unique homomorphism $\beta \colon \mathsf{L}(X) \rightarrow M$ of left $S$-modules
such that $\beta \iota = \alpha$.
Define $\beta (A) = \bigvee_{a \in A} \alpha (a)$.
This definition makes sense, since if $a \sim_{l} b$ then $\alpha (a) \sim_{l} \alpha (b)$.
It is routine to check that $\beta$ has all the required properties.
\end{proof}

\begin{example}\label{ex:hornchurch}
Let $S$ be a pseudogroup and $I$ a left ideal.
There is an action $S \times I \rightarrow I$ by left multiplication.
It is also a supported action when we define $p \colon I \rightarrow \mathsf{E}(S)$ by $p(x) = xx^{-1}$. 
Using Lemma~\ref{lem:scampton}, we may construct a left $S$-module $\mathsf{L}(I)$.
In particular, in this way, there is a left $S$-module $\mathsf{L}(S)$ which makes up for the fact that
the left action of $S$ on itself does not form a left $S$-module.
\end{example}

\section{Characterizations of Morita equivalence}

In this section, we connect the definition of Morita equivalence for pseudogroups described in section~\ref{sec:pseudogroupmodules}
with the Morita theory for inverse quantal frames and \'etale groupoids.

\subsection{Background on quantales and modules}\label{sec:qumod}

Let us recall basic facts and definitions. General references on quantales are~\cite{Qbook,Rosenthal}. For monoids and actions in monoidal categories see~\cite[Ch.~VII]{MacLane}.

First, following~\cite[Ch.~I]{JT}, by a \emph{sup-lattice} will be meant a complete lattice, and a \emph{homomorphism} of sup-lattices $f:L\to M$ is a map that preserves arbitrary joins:
\[
f\bigl(\V_i x_i\bigr) = \V_i f(x_i).
\]
The top element of a sup-lattice $L$ is denoted by $1_L$, or simply $1$, and the bottom element by $0_L$, or simply $0$. The resulting category of sup-lattices $\SL$ is monoidal closed with respect to the tensor product of sup-lattices.

The \emph{category of unital involutive quantales} is the category of involutive monoids in $\SL$. In other words, a unital involutive quantale $Q$ is a sup-lattice equipped with an involutive monoid structure whose multiplication and involution distribute over arbitrary joins,
\[
a \bigl( \bigvee_{i} b_i  \bigr) = \bigvee_i ab_i
\quad\text{ and }\quad
\bigl(\V_i a_i\bigr)^*=\V_i a_i^*,
\]
(multiplication distributes on both sides due to the involution) and a homomorphism of unital quantales is a sup-lattice homomorphism which is also a homomorphism of involutive monoids.
The multiplicative unit of a unital quantale $Q$ is denoted by $e_Q$, or simply $e$. The frames are exactly the unital involutive quantales with idempotent multiplication such that $e=1$ (then, necessarily, $ab=a\wedge b$). Similarly to any commutative quantale, the involution on a frame can be taken to be trivial: $a^*=a$. If a frame carries an independent structure of unital involutive quantale it is referred to as a \emph{quantal frame}.

The \emph{category of left modules} of a unital involutive quantale $Q$ is the category of left $Q$-actions in $\SL$ (the involution plays no role in the definition). The action of an element $a\in Q$ on an element $x$ of a left $Q$-module $X$ is denoted by $ax$.

All these structures can be presented by generators and relations. An exposition of this for sup-lattices and quantales in a form which is convenient for the purposes of this paper can be found in~\cite[\S II.1--2]{gamap}. Let us recall this. The free sup-lattice generated by a set $X$ is the powerset $\mathcal P(X)$, and a presentation of a sup-lattice is a pair $(X,R)$ where $X$ is the set of generators and $R$ is a subset of $\mathcal P(X)\times\mathcal P(X)$; each pair $(S,T)\in R$ is to be thought of as an equation $\V S=\V T$. Concretely, a construction of the sup-lattice presented by $(X, R)$ is given by the following subset of $\mathcal P(X)$:
\[
\langle X\vert R\rangle=\bigl\{Y\subset X : (S\subset Y\iff T\subset Y)\textrm{ for all }(S,T)\in R\bigr\}.
\]
This is closed under arbitrary intersections in $\mathcal P(X)$, so it defines a closure operator, and the injection of generators $\eta:X\to\mathcal \langle X\vert R\rangle$ assigns to each $x\in X$ the closure of $\{x\}$. Then $\eta_X$ is universal among the maps $f:X\to L$ to sup-lattices $L$ that \emph{respect the relations} in $R$, i.e., such that $\V f(S)=\V f(T)$ for all $(S,T)\in R$.

Unital involutive quantales can be presented in a similar way. Letting $M$ be an involutive monoid, the free sup-lattice $\mathcal P(M)$ is a unital involutive quantale under pointwise multiplication and involution, and the singleton assignment $M\to\mathcal P(M)$ is universal among the homomorphisms of involutive monoids $f:M\to Q$ to unital involutive quantales $Q$. Then, if $(M,R)$ is a sup-lattice presentation and $R$ is \emph{stable} --- i.e.,
$(mS,mT)\in R$ and $(S^*,T^*)\in R$ for all $(S,T)\in R$ and $m\in M$ ---,
the presented sup-lattice $\langle M\vert R\rangle$ is a unital involutive quantale. The multiplication of two sets $U,V\in \langle M\vert R\rangle$ is the closure of the pointwise product of $U$ and $V$, and the involution is pointwise. This quantale has the expected universal property, namely every homomorphism of involutive monoids $f:M\to Q$ to a unital involutive quantale $Q$ such that $\V f(S)=\V f(T)$ for all $(S,T)\in R$ extends uniquely to a homomorphism of unital involutive quantales $f^\sharp:\langle M\vert R\rangle\to Q$.

Related to this, let $A$ be a meet-semilattice and consider the set of relations
\[
R=\bigl\{(\{a,b\},\{b\})\in\mathcal P(A)\times\mathcal P(A) : a\le b\bigr\}.
\]
Then $\langle A\vert R\rangle$ is the frame $\lc(A)$ whose elements are the downwards closed sets of $A$, and it is universal among homomorphisms of meet semi-lattices $f:A\to F$ to frames $F$. The injection of generators is the principal ideal mapping $a\mapsto a^\downarrow$. If we add further relations that are stable for $\wedge$ we obtain again a frame, and a corresponding universal property. See~\cite[\S II.2.11]{J}.

The above facts are easily adapted to presentations of quantale modules. Let us describe this.

\begin{definition*}
Let $M$ be an involutive monoid acting on the left on a set $X$, and let
\[
R_M\subset\mathcal P(M)\times\mathcal P(M)\quad\textrm{ and }\quad R_X\subset\mathcal P(X)\times\mathcal P(X).
\]
We say that the pair $(R_M,R_X)$ is \emph{jointly stable} if the following conditions hold:
\begin{enumerate}
\item $m R_M\subset R_M$ and $R_M^*= R_M$ for all $m\in M$\\
(i.e., $R_M$ is stable as described above);
\item $m R_X\subset R_X$ for all $m\in M$;
\item $R_M x\subset R_X$ for all $x\in X$.
\end{enumerate}
\end{definition*}

\begin{proposition}\label{prop:presentations}
Let $M$ be an involutive monoid acting on the left on a set $X$, and let $(R_M,R_X)$ be a jointly stable pair as in the above definition.
\begin{enumerate}
\item The sup-lattice $\langle X\vert R_X\rangle$ is a left $\langle M\vert R_M\rangle$-module whose action uniquely extends the action of $M$ on $X$.
\item Given any left $\langle M\vert R_M\rangle$-module $\Xi$ and any map $f:X\to \Xi$ that respects the relations in $R_X$ and is $M$-equivariant (i.e., $f(mx) = \eta_M(m) f(x)$ for all $m\in M$ and $x\in X$), there is a unique homomorphism of left $\langle M\vert R_M\rangle$-modules $f^\sharp:\langle X\vert R_X\rangle\to\Xi$ such that the following diagram commutes:
\[
\xymatrix{
X\ar[drr]_f\ar[rr]^{\eta_X}&&\lcc(X)\ar[d]^{f^\sharp}\\
&&\Xi.
}
\]
Explicitly, $f^\sharp$ is defined for all $J\in\langle X\vert R_X\rangle$ by
\[
f^\sharp(J) = \V_{x\in J} f(x).
\]
\end{enumerate}
\end{proposition}

A proof of this proposition is presented in the appendix.

\subsection{Background on inverse quantal frames}

Recall that an inverse quantal frame \cite{Resende2007} is a unital involutive quantale that is isomorphic to the quantale $\opens(G)$ of a localic \'etale groupoid $G$. There is a concrete construction of an \'etale groupoid $\mathsf{G}(Q)$ from an inverse quantal frame $Q$, and there are isomorphisms $G\cong\mathsf{G}(\opens(G))$ and $Q\cong\opens(\mathsf{G}(G))$.

Using the simpler characterization of \cite{Resende2018}, an inverse quantal frame is the same as a unital involutive quantal frame $Q$ that satisfies the following two conditions:
\begin{enumerate}
\item $aa^*a\le a\Longrightarrow a=aa^*a$ for all $a\in Q$ (i.e., $Q$ is \emph{stably Gelfand});
\item $1=\V Q_\ipi$, where $Q_\ipi=\{s\in Q\st ss^*\le e,\ s^*s\le e\}$
is the subset of \emph{partial units}, which is necessarily a pseudogroup.
\end{enumerate}
We borrow notation and terminology from \cite{Resende2012} and denote the \emph{base locale} of an arbitrary inverse quantal frame $Q$ by $Q_0 = \{a\in Q\st a\le e\}$. 
This coincides with the locale of idempotents $\mathsf{E}(Q_\ipi)$.

If $S$ is a pseudogroup, then $\lcc(S)$ is its associated inverse quantal frame,  which consists of all order-ideals of $S$ that are closed under joins of compatible sets. 
We have $\lcc(S)_\ipi\cong S$ and $\lcc(Q_\ipi)\cong Q$ for all pseudogroups $S$ and inverse quantal frames $Q$. 
Both $(-)_\ipi$ and $\lcc$ extend to functors and show that the categories of inverse quantal frames and of pseudogroups are equivalent.

If $S$ is a pseudogroup, we shall write $\mathsf{G}(S)$ for the localic \'etale groupoid $\mathsf{G}(\lcc(S))$, and call it the \emph{groupoid of $S$}.
If $S$ is spatial then $\mathsf{G}(S)$ can be regarded as a sober topological \'etale groupoid. Even if $S$ is not spatial, we can always obtain a sober topological \'etale groupoid from $S$ by taking the locale spectrum of $\mathsf{G}(S)$. This construction yields the groupoid of completely prime filters of $S$ \cite{MR,LL}.

\begin{remark}\label{rem:lccS}
The inverse quantal frame $\lcc(S)$ of a pseudogroup $S$ coincides with the sup-lattice presented as
$\langle S\vert R_S\rangle$, where $R_S$ contains the relations $(C,\{\V C\})$ for all compatible sets $C\subset S$. The injection of generators $\eta_S:S\to\lcc(S)$ is the principal ideal mapping $s\mapsto s^{\downarrow}$. The distributivity properties of $S$ imply that $R_S$ is stable both with respect to the involutive monoid structure and the meets of $S$, and thus $\lcc(S)$ has several universal properties, as explained in section~\ref{sec:qumod}: as a sup-lattice; as a unital involutive quantale; and as a frame. This kind of reasoning will be useful below when dealing with pseudogroup modules.
\end{remark}

\subsection{Sheaves and Morita equivalence for inverse quantal frames}\label{sec:shmoritaiqf}

An \emph{equivariant sheaf} on a localic \'etale groupoid $G$, or a \emph{$G$-sheaf}, is a local homeomorphism $p \colon X \to G_0$, where $G_0$ is the locale of objects of the groupoid, 
together with a left action of $G$ on $X$ satisfying properties such as associativity, the details of which need not concern us here. 
The category of such sheaves and equivariant maps between them is a Grothendieck topos known as the \emph{classifying topos of $G$}, usually written $\mathsf{B}G$.
Regarding this as the ``category of left modules of $G$,'' we formulate one of several equivalent definitions of Morita equivalence for \'etale groupoids:

\begin{definition*}
Two \'etale groupoids $G$ and $H$ are \emph{Morita equivalent} if $\mathsf{B}G$ and $\mathsf{B}H$ are equivalent categories.
\end{definition*}

\begin{remark}\label{rem:HSmaps}
For non-\'etale groupoids in general, Morita equivalence cannot be defined via classifying toposes.
For instance, Morita equivalence of Lie groupoids is the isomorphism relation in the category whose morphisms are the isomorphism classes of principal bi-bundles~\cite{MoerMrcun}, also known as Hilsum--Skandalis maps, and the classifying topos is only a Morita invariant. But the analogous definition for localic groupoids coincides with equivalence of classifying toposes~\cite{Bunge} when restricted to the class of \'etale-complete groupoids, which contains the \'etale groupoids.
\end{remark}

In this paper it will suffice to work with the quantale-theoretic characterizations of $\mathsf{B}G$ given in \cite{Resende2012}. There it is shown that the sheaves on an \'etale groupoid $G$ can be precisely described in terms of modules on the inverse quantal frame $\opens(G)$, and that $\mathsf{B}G$  can be represented as a category of $\opens(G)$-modules. 
There are in fact two equivalent main definitions of such modules, the first of which we shall recall below and then use in section~\ref{sec:equivcat}. The second definition, in terms of Hilbert modules, will only be recalled and used in section~\ref{sec:morebackground}.

The first definition is scattered across several places in \cite{Resende2012}, and here we present it in a way which is convenient for the purposes of this paper:

\begin{definition*}\cite{Resende2012}
Let $Q$ be an inverse quantal frame. By a \emph{$Q$-sheaf} is meant a left $Q$-module $X$ which is also a locale, equipped with a (necessarily unique) monotone map $\spp:X\to Q_0$, called the \emph{support}, such that
\begin{enumerate}
\item $\spp(bx)=b\spp(x)$ for all $b\in Q_0$ and $x\in X$ ($Q_0$-equivariance),
\item $\spp(x)x=x$ for all $x\in X$ (support condition),
\item $\V\sections_X=1_X$ (cover condition),
\end{enumerate}
where $\sections_X$ is the set of \emph{local sections}:
\[
\sections_X=\bigl\{\gamma\in X\st \spp(x\wedge \gamma)\gamma\le x \text{ for all }x\in X\bigr\}.
\]
Given $Q$-sheaves $X$ and $Y$, a \emph{sheaf homomorphism} $\phi:X\to Y$ is a homomorphism of left $Q$-modules that commutes with the supports and sends local sections to local sections:
\[
\spp_Y\circ\phi=\spp_X\quad\textrm{and}\quad
\phi(\sections_X)\subset\sections_Y.
\]
The $Q$-sheaves together with the sheaf homomorphisms form the \emph{category of $Q$-sheaves}, which we denote by $Q$-$\Sh$.
\end{definition*}

From the definition of local section it immediately follows that the set $\sections_X$ is downwards closed in $X$, and that for all $\gamma,\gamma'\in\sections_X$ we have (\cf\ \cite[Lemma~2.35]{Resende2012})
\begin{equation}\label{localsectionprop}
\gamma\le\gamma'\iff\gamma=\spp(\gamma)\gamma'.
\end{equation}
We also remark that the support of a $Q$-sheaf is unique, and that it preserves arbitrary joins, because the conditions that it satisfies make it left adjoint to the mapping $(-)1_X:Q_0\to X$. Hence, being a $Q$-sheaf is a property of a left $Q$-module, rather than extra structure. In addition, a property that will be useful is the following~\cite[Th. 3.25]{Resende2012}: for all $s\in Q_\ipi$ and $x\in X$ we have
\begin{equation}\label{conjugation}
\spp(sx)=s\spp(x)s^{-1},
\end{equation}
and thus also
\begin{equation}\label{conjugation2}
s^{-1}\spp(sx)=\spp(x)s^{-1}.
\end{equation}

\begin{theorem}[{\cite[Th.~3.33]{Resende2012}}]\label{th:QShvsBG}
Let $G$ be an \'etale groupoid. The category of $\opens(G)$-sheaves $\opens(G)$-$Sh$ is isomorphic to the classifying topos $\mathsf{B}G$.
\end{theorem}

\begin{corollary}\label{cor:morita}
Let $Q$ and $R$ be inverse quantal frames. The following conditions are equivalent:
\begin{enumerate}
\item The categories $Q$-$\Sh$ and $R$-$\Sh$ are equivalent;
\item The \'etale groupoids $\mathsf{G}(Q)$ and $\mathsf{G}(R)$ are Morita equivalent.
\end{enumerate}
\end{corollary}

Accordingly, we adopt the following definition (\cf\ \cite[\S 5.6]{QR}):

\begin{definition*}
Two inverse quantal frames are \emph{Morita equivalent} if their categories of sheaves are equivalent.
\end{definition*}

\subsection{Morita equivalence for pseudogroups}\label{sec:equivcat}

The purpose of this section is to prove that the category of modules of an arbitrary pseudogroup $S$ is equivalent to the category of sheaves of its inverse quantal frame $\lcc(S)$. As a consequence, two pseudogroups are Morita equivalent if and only if their inverse quantal frames are, and this proves Theorem~\ref{theoremA} of the Introduction.

The first step will be to provide a construction of pseudogroup modules from sheaves on inverse quantal frames:

\begin{lemma}\label{lem:Theta}
Let $Q$ be an inverse quantal frame, and $X$ a $Q$-sheaf. Then $\sections_X$ is a $Q_\ipi$-module whose action is the restriction of the action of $Q$ to partial units and local sections, and whose support is the restriction of the support of $X$ to local sections. Moreover, any homomorphism of $Q$-sheaves $\phi:X\to Y$ restricts to a homomorphism of $Q_\ipi$-modules $\phi\vert_{\sections_X}:\sections_X\to\sections_Y$, thus yielding a functor
\[
\Theta:Q\textrm{-}\Sh\to Q_\ipi\textrm{-}\Mod.
\]
\end{lemma}

\begin{proof}
The action of $Q_\ipi$ on $\sections_X$ is well defined due to \cite[Lemma 4.52]{Resende2012}. In order to make the paper more self-contained let us provide an explicit (and slightly shorter) proof of this fact here. Let $s\in Q_\ipi$ and $\gamma\in\sections_X$. We need to show that $s\gamma$ is a local section, i.e., that the inequality
\[
\spp(x\wedge s\gamma)s\gamma\le x
\]
is satisfied for all $x\in X$. Indeed, multiplying by $s^{-1}$ on the left in the inequality $x\wedge s\gamma\le s\gamma$ we obtain
\[
s^{-1}(x\wedge s\gamma)\le\gamma,
\]
so due to \eqref{localsectionprop} we have
\[
\spp\bigl(s^{-1}(x\wedge s\gamma)\bigr)\gamma=s^{-1}(x\wedge s\gamma),
\]
and, using \eqref{conjugation2},
\[
\spp(x\wedge s\gamma)s\gamma = s\spp\bigl(s^{-1}(x\wedge s\gamma)\bigr)\gamma
=ss^{-1}(x\wedge s\gamma)\le x.
\]
So there is a well defined left $Q_\ipi$-action on $\sections_X$, which moreover is supported, with support $p:\sections_X\to E(Q_\ipi)$ equal to $\spp$ restricted to $\sections_X$ --- the two defining conditions of the support of a supported left $Q_\ipi$-action follow, respectively, from the support condition of the definition of $Q$-sheaf, and from \eqref{conjugation}.

Since $X$ is a left $Q$-module, the distributivity conditions (3) and (4) of the definition of left $Q_\ipi$-module (\cf\ section~\ref{sec:modules}) are automatically satisfied, so in order to show that $\sections_X$ is a $Q_\ipi$-module it only remains to be proved that $\sections_X$ is closed under joins of left compatible sets. Let $Y\subset \sections_X$ be a left compatible set, and let $\gamma=\V Y$ in $X$. All we need is to show that $\gamma$ is a local section, so let us pick an arbitrary element $x\in X$ and prove that $\spp(x\wedge\gamma)\gamma\le x$. Since both the action and the binary meets in $X$ distribute over joins, and $\spp$ preserves joins, we obtain
\[
\spp(x\wedge\gamma)\gamma = \spp\bigl(x\wedge\V_{y\in Y} y\bigr)\V_{z\in Y}z=\V_{y\in Y}\V_{z\in Y}\spp(x\wedge y)z,
\]
so the desired conclusion will follow from showing that $\spp(x\wedge y)z\le x$ for all left compatible elements $y,z\in\sections_X$. First notice that for all $b\in Q_0$ we have
\[
b(x\wedge y)=x\wedge by,
\]
by a routine calculation: from the properties of meets we obtain
\[
b(x\wedge y)\le bx\wedge by\le x\wedge by,
\]
and, conversely, since $x\wedge by$ is a local section below $by$ and $y$,
\[
x\wedge by=\spp(x\wedge by)by=b\spp(x\wedge by)y= b(x\wedge by)\le b(x\wedge y).
\]
Therefore, writing $w$ for the local section $\spp(z)y=\spp(y)z$, and using the $Q_0$-equivariance of $\spp$, we obtain
\[
\spp(x\wedge y)z=\spp\bigl(x\wedge\spp(y)y\bigr)\spp(z)z=\spp\bigl(x\wedge\spp(z)y\bigr)\spp(y)z=\spp(x\wedge w)w\le x.
\]
Hence, $\sections_X$ is a left $Q_\ipi$-module.

Finally, if $\phi:X\to Y$ is a sheaf homomorphism then by definition the restriction $\phi\vert_{\sections_X}:\sections_X\to\sections_Y$ is well defined and $Q_\ipi$-equivariant, and it commutes with the supports, so it is a homomorphism of left $Q_\ipi$-actions; and, being a homomorphism of $Q$-modules, it preserves all the existing joins, so it is a homomorphism of $Q_\ipi$-modules. The assignments $X\mapsto\sections_X$ and $\phi\mapsto\phi\vert_{\sections_X}$ are respectively the object part and the morphism part of the required functor $\Theta:Q\textrm{-}\Sh\to Q_\ipi\textrm{-}\Mod$.
\end{proof}

Now let $S$ be a pseudogroup.
Analogously to the definition of the inverse quantal frame $\lcc(S)$, 
for each left $S$-module $X$ 
we shall denote by $\lcc(X)$ the set of all order-ideals of $X$ that are closed under left compatible joins. This coincides with the sup-lattice presented as
$\langle X\vert R_X\rangle$, where $R_X$ contains the relations $(C,\{\V C\})$ for all left compatible sets $C\subset X$, and the injection of generators $\eta_X:X\to\lcc(X)$ is the principal ideal mapping $x\mapsto x^{\downarrow}$ (cf.\ Remark~\ref{rem:lccS}).

\begin{lemma}\label{lem:sheavesfrommodules}
Let $S$ be a pseudogroup, and $X$ an $S$-module with support $p$. Then $\lcc(X)$ is an $\lcc(S)$-sheaf whose local sections are the principal ideals. The left $\lcc(S)$-action is the unique extension of the $S$-action on $X$, and the support $\spp$ is the unique extension of $p$; that is, for all $U\in\lcc(S)$ and $Z\in\lcc(X)$,
\[
UZ = \V_{u\in U}\V_{z\in Z}(uz)^{\downarrow}\quad\textrm{ and }\quad
\spp(Z) = \V_{z\in Z} p(z)^{\downarrow}.
\]
\end{lemma}

\begin{proof}
The distributivity properties of the pseudogroup action imply that the sets of defining relations of $\lcc(S)$ and $\lcc(X)$ are jointly stable (cf.\ section~\ref{sec:qumod}), so the conclusion that $\lcc(X)$ is an $\lcc(S)$-module, together with the formula for the action, follows from Proposition~\ref{prop:presentations}. Moreover, by Lemma~\ref{lem:harpo} $X$ is also a meet-semilattice, and, due to the distributivity of binary joins over left compatible joins, the defining relations of $\lcc(X)$ are stable with respect to $\wedge$, so $\lcc(X)$ is a frame.

Now let us study the support of $\lcc(X)$.
Since $p:X\to E(S)$ preserves left compatible joins, the sup-lattice universal property of $\lcc(X)$ extends $p$ to a unique sup-lattice homomorphism $\spp:\lcc(X)\to E(S)\cong\lcc(S)_0$ given by the formula in the statement.
Let us verify that this is $\lcc(S)_0$-equivariant. If $J\in\lcc(S)_0$ and $I\in\lcc(X)$ then, since $J\subset E(S)$ and $p$ is $E(S)$-equivariant,
\begin{eqnarray*}
\spp(JI) &=& \spp\bigl(\V_{s\in J}\V_{x\in I} (sx)^\downarrow\bigr)
=\V_{s\in J}\V_{x\in I} \spp\bigl((sx)^\downarrow\bigr)
=\V_{s\in J}\V_{x\in I} p(sx)^\downarrow\\
&=&\V_{s\in J}\V_{x\in I} (sp(x))^\downarrow=J\spp(I).
\end{eqnarray*}
Let us verify that $\spp$ satisfies the support condition of a sheaf.
Let $J\in\lcc(X)$. 
Taking into account the inequality $p(x)y\le y$ for all $x,y\in X$, we have
\[
\spp(J)J = \V_{x,y\in J}(p(x)y)^{\downarrow} =\V_{x\in J} (p(x)x)^{\downarrow}=\V_{x\in J} x^{\downarrow} = J.
\]

In order to conclude that $\lcc(X)$ is a sheaf only the cover condition remains to be proved. For this it suffices to show that each principal ideal $x^{\downarrow}$ is a local section.
Let $x\in X$ and $J\in\lcc(X)$, and let us prove the local section condition $\spp(J\cap x^{\downarrow})x^{\downarrow}\subset J$.
Note that $J\cap x^{\downarrow}$ is upper bounded by $x$, and thus it is a left compatible set. 
Therefore $J\cap x^{\downarrow}$ is a principal ideal $y^{\downarrow}$ with $y\le x$. Hence,
\[
\spp(J\cap x^{\downarrow}) x^{\downarrow}
= \spp(y^{\downarrow}) x^{\downarrow}
= (p(y)x)^{\downarrow}
= y^{\downarrow} \subset J .
\]

Finally, let us prove that every local section is a principal ideal.
Let $\gamma\in\lcc(X)$ be a local section.
For all $x\in\gamma$ we have $x^{\downarrow}\subset\gamma$, so by \eqref{localsectionprop} we obtain $x^{\downarrow} = \spp(x^{\downarrow})\gamma$. Let $x,y\in\gamma$. 
Then
\[
(p(x)y)^{\downarrow} = \spp(x^{\downarrow}) y^{\downarrow}
=\spp(x^{\downarrow})\spp(y^{\downarrow})\gamma
=\spp(y^{\downarrow})\spp(x^{\downarrow})\gamma
= \spp(y^{\downarrow}) x^{\downarrow} = (p(y)x)^{\downarrow}.
\]
This shows that $\gamma$ is a left compatible set, so it is the principal ideal $(\V\gamma)^{\downarrow}$.
\end{proof}

\begin{theorem}\label{them:one} Let $S$ be a pseudogroup.
The categories $S$-$\Mod$ and $\lcc(S)$-$\Sh$ are equivalent. Hence, in particular, $S$-$\Mod$ is a Grothendieck topos. 
\end{theorem}
\begin{proof}
Recall the functor $\Theta:\lcc(S)$-$\Sh\to\lcc(S)_\ipi$-$\Mod$ of Lemma~\ref{lem:Theta}. Since the principal ideal mapping $(-)^\downarrow:S\to\lcc(S)_\ipi$ is an isomorphism of pseudogroups, let us regard $\Theta$ as a functor from $\lcc(S)$-$\Sh$ to $S$-$\Mod$. In particular, for an $\lcc(S)$-sheaf $\Xi$ with support $\tau:\Xi\to\lcc(S)_0$, the support of the $S$-module $\Theta(\Xi)$ is the map $q:\sections_\Xi\to E(S)$ which for all $\gamma\in\sections_\Xi$ is given by
\begin{equation}\label{suppeq}
\tau(\gamma) = q(\gamma)^\downarrow.
\end{equation}

Now let $X$ be an $S$-module with support $p$, and $\Xi$ an $\lcc(S)$-sheaf as above. Let also $f:X\to\Theta(\Xi)$ be a homomorphism of $S$-modules. Due to the universal property of $\lcc(X)$ as a left $\lcc(S)$-module (which follows from the distributivity of the action of $X$ over joins), there is a unique homomorphism of left $\lcc(S)$-modules $f^\sharp:\lcc(X)\to\Xi$ such that $f^\sharp(x^\downarrow)=f(x)$ for all $x\in X$ (\cf\ Proposition~\ref{prop:presentations}); for all $J\in\lcc(S)$ this is given by
\begin{equation}\label{fsharp}
f^\sharp(J)=\V_{x\in J} f(x).
\end{equation}
Moreover, $f^\sharp$ sends local sections to local sections because the image of $f$ is in $\Theta(\Xi)$, and $f^\sharp$ commutes with the supports of $\lcc(X)$ and $\Xi$ because for all $J\in\lcc(X)$ we have, using \eqref{suppeq} and \eqref{fsharp} and writing $\spp$ for the support of $\lcc(X)$:
\begin{eqnarray*}
\tau(f^\sharp(J))&=&\tau\bigl(\V_{x\in J}f(x)\bigr)
=\V_{x\in J}\tau(f(x))
=\V_{x\in J} q(f(x))^\downarrow\\
&=&\V_{x\in J} p(x)^\downarrow
=\V_{x\in J} \spp(x^\downarrow)
=\spp(J).
\end{eqnarray*}
Hence, $f^\sharp$ is a homomorphism of $\lcc(S)$-sheaves, and it is the unique such that makes the following diagram commute:
\[
\xymatrix{
X\ar[rr]^-{\eta_X}\ar[drr]_-{f}&&\Theta(\lcc(X))\ar[d]^-{\Theta(f^{\sharp})}\\
&&\Theta(\Xi).
}
\]
So the construction $\lcc(-)$ is the object part of a functor from $S$-$\Mod$ to $\lcc(S)$-$\Sh$ which is left adjoint to $\Theta$, and the adjunction is a co-reflection because $\eta_X$ is an isomorphism of $S$-modules for all $S$-modules $X$.

In order to conclude that $\Theta$ is an equivalence of categories all that remains to be proved is that the co-unit of the adjunction is an isomorphism. For each $\lcc(S)$-sheaf $\Xi$ the component
$\epsilon_\Xi:\lcc(\Theta(\Xi))\to\Xi$ of the co-unit is the extension $\ident^\sharp$ of the identity homomorphism of $S$-modules $\ident:\Theta(\Xi)\to\Theta(\Xi)$. So $\epsilon_\Xi$ is a homomorphism of $\lcc(S)$-sheaves and, moreover, since the identity map preserves meets, the frame universal property of $\lcc(\Theta(\Xi))$ implies that $\epsilon_\Xi$ is also a frame homomorphism (\cf\ section~\ref{sec:qumod}). The principal ideals of $\lcc(\Theta(\Xi))$ form a basis in the frame sense, so $\epsilon_\Xi$ is  surjective. Moreover, the basis is downwards closed, and the restriction of $\epsilon_\Xi$ to the basis is injective, so $\epsilon_\Xi$ is a frame isomorphism due to~\cite[Prop.~2.2]{Resende2007}.
\end{proof}
 
So we have arrived at Theorem~\ref{theoremA} from the Introduction:

\begin{corollary}\label{them:groupoids}
Let $S$ and $T$ be pseudogroups. 
The following conditions are equivalent:
\begin{enumerate}
\item $S$ and $T$ are Morita equivalent pseudogroups.
\item $\lcc(S)$ and $\lcc(T)$ are Morita equivalent inverse quantal frames.
\item $\mathsf{G}(S)$ and $\mathsf{G}(T)$ are Morita equivalent \'etale groupoids.
\end{enumerate}
\end{corollary}

\section{Bimodules}

In this section, we develop further algebraic tools needed to show that two pseudogroups are Morita equivalent.

\subsection{Background on biactions}
Let $S$ and $T$ be inverse semigroups.
By an {\em $(S,T)$-biaction} we will mean a set $X$ equipped with a left action $S \times X \rightarrow X$ and a right action $X \times T \rightarrow X$,
such that $(sx)t = s(xt)$ for all $s \in S$, $t \in T$ and $x \in X$,
together with two functions, called {\em inner products}, $\langle -,- \rangle \colon X \times X \rightarrow S$ and $[-,-] \colon X \times X \rightarrow T$ such that the following axioms hold
for all $x,y,z \in X$, $s \in S$ and $t \in T$:
\begin{description}
\item[{\rm (BA1)}] $\langle sx, y \rangle = s \langle x,y \rangle$.
\item[{\rm (BA2)}]  $\langle y, x \rangle = \langle x,y \rangle^{-1}$.
\item[{\rm (BA3)}]  $\langle x, x \rangle x = x$. 
\item[{\rm (BA4)}]  $[x,yt] = [x,y]t$.
\item[{\rm (BA5)}]  $[y,x] = [x,y]^{-1}$.
\item[{\rm (BA6)}]  $x[x,x] = x$.
\item[{\rm (BA7)}] $\langle x, y \rangle z = x[y,z]$.
\end{description}

The above definition is the same as Steinberg's definition of a Morita context to be found in \cite[Definition~2.1]{Steinberg}
{\em except that we do not require the functions $\langle -,- \rangle$ and $[-,-]$ to be surjective}.
We shall need the following properties expressing how the above inner products interact with the action.
They are all proved in \cite[Prop.~2.3]{Steinberg}.

\begin{proposition}\label{lem:biaction-action}
In a biaction, the following properties hold:
\begin{enumerate}
\item $\langle x, sy \rangle = \langle x, y \rangle s^{-1}$.
\item $[xt,y] = t^{-1}[x,y]$.
\item $\langle xt, y \rangle = \langle x, yt^{-1} \rangle$.
\item $[x,sy] = [s^{-1}x,y]$.
\end{enumerate}
\end{proposition}

\begin{remark}
Observe that $\mbox{\rm im}\langle -,- \rangle$ is an ideal of $S$;
this follows by (BA1) above, that tells us it is a left ideal, and by Proposition~\ref{lem:biaction-action}, that tells us it is a right ideal.
Similarly, $\mbox{\rm im}[-,-]$ is an ideal of $T$.
\end{remark}

Observing that both $[x,x]$ and $\langle x, x \rangle$ are idempotents, we obtain:

\begin{lemma}\label{surjectivia}
Let $S$ and $T$ be inverse semigroups. 
Any $(S,T)$-biaction $X$ is both a supported left $S$-action and a supported right $T$-action with supports $p\colon X\to\mathsf{E}(S)$ and $q\colon X\to\mathsf{E}(T)$ defined for all $x\in X$ by
$$ p(x)=\langle x,x\rangle\quad\quad\textrm{and}\quad\quad q(x)=[x,x]\;. $$
Moreover, the following hold:
\begin{enumerate}
\item $\mbox{\rm im}\langle -,- \rangle=S$ if and only if $p$ is surjective.
\item $\mbox{\rm im}[-,-]=T$ if and only if $q$ is surjective.
\item If $\langle x,y \rangle$ and $[x,y]$ are both idempotents then $x \sim y$.
\item $q(x)[x,y] = [x,y]$, and dually.
\end{enumerate}
\end{lemma}

\begin{proof}
All these points are easily derived from  \cite{Steinberg}. Let us just address (3) and (4). In order to prove (3) let us assume that both $\langle x,y \rangle$ and $[x,y]$ are idempotents.
From the properties, it is easy to prove that
$$\langle x, y \rangle \langle x, y \rangle = \langle x, x \rangle \langle y,y \rangle$$
and
$$[x, y][x, y] = [x, x][y,y].$$
Then
$$p(x)y = \langle x, x \rangle y = x[x,y] = x[y,x]$$
since $[x,y]$ is an idempotent.
But
$$x[y,x] = \langle x, y \rangle x = \langle x, x \rangle \langle y, y \rangle x = \langle y, y \rangle x=p(y)x.$$
Thus $p(x)y = p(y)x$, so we have proved that $x \sim_{l} y$.
By symmetry, $x \sim_{r} y$, thus $x \sim y$. This proves (3), and (4) is a consequence of part (4) of Proposition~\ref{lem:biaction-action}.
\end{proof}

\begin{remark}
Steinberg's theory \cite{Steinberg} is developed for inverse semigroups.
In this paper, we are primarily interested in inverse semigroups with zero.
His theory can be extended to this case by using pointed actions.
\end{remark}

\subsection{Pseudogroup bimodules}
Let $S$ and $T$ be pseudogroups.
By an {\em $(S,T)$-bimodule} we will mean an $(S,T)$-biaction $X$ such that the following axioms hold:
\begin{description}
\item[{\rm (Completeness)}] Every \emph{compatible} subset of $X$ (that is, a subset which is both left and right compatible) has a join.
\item[{\rm (Distributivity I)}] Both actions distribute over compatible joins in each variable.
That is, for all $x\in X$, $s\in S$, and $t\in T$, and all compatible subsets $\{x_i \st i\in I\}$ of $X$, $\{s_i \st i\in I\}$ of $S$, and $\{t_i\st i\in I\}$ of $T$, we have
that the following conditions hold:
\begin{enumerate}
\item $s \left( \V_{i \in I} x_{i} \right) = \V_{i \in I} sx_{i}$.
\item $\left( \V_{i \in I} s_{i} \right)x = \V_{i \in I} s_{i}x$.
\item $\left( \V_{i \in I} x_{i} \right)t = \V_{i \in I} x_{i}t$.
\item $x\left( \V_{i \in I} t_{i} \right) = \V_{i \in I} xt_{i}$.
\end{enumerate}
\item[{\rm (Distributivity II)}] For any compatible subset $\{x_{i} \st i \in I\}$ of $X$ we have that
$$\bigl\langle \V_{i\in I} x_{i}, y \bigr\rangle = \V_{i \in I} \langle x_{i}, y \rangle 
\text{ and } 
\bigl[y, \V_{i\in I} x_{i}\bigr] = \V_{i \in I} [y, x_{i}].$$ 
In particular, $\langle 0,y \rangle = 0$ and $[y,0] = 0$.
\end{description}

\begin{remark}\label{rem:collision}
The reader should be aware that if $X$ is an $(S,T)$-bimodule it need be neither a
left $S$-module nor a right $T$-module. This is because we require that every compatible subset of $X$ has a join
but we do not require that every left (respectively, right) compatible subset has a join. This leads to an unfortunate collision of terminology, but it is nothing more than that.
\end{remark}

The two conditions in the righthand side of the equivalences in the following lemma 
mean that the supports of the elements of $X$ cover the pseudogroup units, so we refer to them as \emph{covering conditions} (cf.\ Lemma~\ref{surjectivia}).

\begin{lemma}[Covering conditions]
\label{lem:coveringconditions}
Let $S$ and $T$ be pseudogroups, with identities $e_{S}$ and $e_{T}$, respectively, and $X$ an $(S,T)$-bimodule.
The following hold:
\begin{enumerate}
\item $\left(\mbox{\rm im}\langle -,-\rangle\right)^{\V} = S$ if and only if $\V_{x\in X} p(x)=e_S$.
\item $\left(\mbox{\rm im}[-,-] \right)^{\V}=T$ if and only if $\V_{x\in X} q(x)=e_T$.
\end{enumerate}
\end{lemma}
\begin{proof} We prove (1), the proof of (2) is similar.
Only one direction needs proving.
Suppose that $\V_{x\in X} p(x)=e_S$.
Let $s \in S$.
Then 
$$s = \V_{x\in X} sp(x) =  \V_{x\in X} s\langle x, x \rangle = \V_{x\in X} \langle sx, x \rangle,$$
as required.
\end{proof}

\begin{definition*}
Let $S$ and $T$ be pseudogroups.
An {\em $(S,T)$-equivalence bimodule} is an $(S,T)$-bimodule satisfying both the above covering conditions.
\end{definition*}

The following example shows that the notion of an equivalence bimodule arises naturally in the theory of transformation pseudogroups.

\begin{example}\label{ex:atlases}
Let $X$ and $Y$ be topological spaces.
Let $S$ be a pseudogroup of partial homeomorphisms on $X$, 
let $T$ be a pseudogroup of partial homeomorphisms on $Y$
and let $\mathcal{X}$ be a set of partial homeomorphisms from $Y$ to $X$
closed under compatible unions.
We now require the following conditions to hold:
\begin{enumerate}
\item $S \mathcal{X} = \mathcal{X}$.
\item $\mathcal{X}T = \mathcal{X}$.
\item $\mathcal{X}\mathcal{X}^{-1} \subseteq S$ and $\mathcal{X}^{-1}\mathcal{X} \subseteq T$.
\item The join of the images of the elements of $\mathcal{X}$ is $X$
and the join of the domains of the elements of $\mathcal{X}$ is $Y$.
\end{enumerate}
Define $\langle x, y \rangle = xy^{-1} \in S$ and define $[x,y] = x^{-1}y \in T$.
Then $\mathcal{X}$ is an equivalence $(S,T)$-bimodule.
However, $\mathcal{X}$ can also be viewed as being an {\em atlas} from $Y$ to $X$
in the usual sense of differential geometry.
You can find a discussion in \cite[\S 1.2]{Lawson1998}.
\end{example}

\begin{example}
Let $X$ be an equivalence $(S,T)$-bimodule.
Define a ternary operation $\{-,-,-\}$ on $X$ by $\{x,y,z\} = \langle x, y \rangle z$.
Then $(X,\{-,-,-\})$ is a generalized heap \cite{Lawson2011}.
It would be interesting to investigate just which generalized heaps arise in this way
in the light of Wagner's work \cite{HL}.
\end{example}

\subsection{Morita equivalence for inverse quantal frames revisited}\label{sec:morebackground}

A characterization of Morita equivalence of inverse quantal frames in terms of bimodules has been obtained in~\cite{QR}, based on a definition of sheaf for inverse quantal frames which is different from but equivalent to that of section~\ref{sec:shmoritaiqf}. The second definition of sheaf is formulated in terms of the Hilbert modules of~\cite{P} (a quantale version of Hilbert C*-modules), specifically the complete ones that were introduced in~\cite{RR2010,Resende2012}. Let us recall this.

\begin{definition*}
Let $Q$ be an involutive quantale. By a \emph{complete Hilbert left $Q$-module} is meant a left $Q$-module $X$ equipped with a map
\[
\langle -,-\rangle:X\times X\to Q,
\]
called the \emph{inner product},
that satisfies the following three conditions for all $x, y\in X$, all families $(x_i)_{i\in I}$ in $X$, and all $a\in Q$,
\begin{enumerate}
\item $\langle{ax},y\rangle=a\langle x, y\rangle$,
\item $\bigl\langle\V_{i\in I} x_i,y\bigr\rangle =\V_{i\in I}\langle{x_i},y\rangle$,
\item $\langle x, y\rangle=\langle y, x\rangle^*$,
\end{enumerate}
and moreover is such that for some subset $\sections\subset X$, called a \emph{Hilbert basis}, the following condition holds for all $x\in X$:
\begin{enumerate}
\setcounter{enumi}{3}
\item $x = \V_{\gamma\in\sections} \langle x,\gamma\rangle\gamma$.
\end{enumerate}
By a \emph{Hilbert section} is meant an element $\gamma\in X$ such that $\langle x,\gamma\rangle\gamma\le x$ for all $x\in X$, and a \emph{regular section} is a Hilbert section $\gamma$ such that $\langle\gamma,\gamma\rangle\gamma=\gamma$. The set of all the Hilbert sections is denoted by $\sections_X$.
\end{definition*}

The following facts are easy consequences of the definition:

\begin{proposition}\label{parseval}
Let $Q$ be an involutive quantale, and $X$ a complete left Hilbert $Q$-module.  
\begin{enumerate}
\item For all $x,y\in X$ and any Hilbert basis $\sections$ we have $\langle x,y\rangle = \V_{\gamma\in\sections} \langle x,\gamma\rangle\langle\gamma,y\rangle$ (Parseval's identity).
\item If $\langle x,z\rangle=\langle y,z\rangle$ for all $z\in X$ then $x=y$ (the inner product is non-degenerate).
\item Any element of a Hilbert basis is a Hilbert section, and $\sections_X$ is the greatest Hilbert basis.
\item Any join-dense set of regular sections is a Hilbert basis.
\end{enumerate}
\end{proposition}

In~\cite{HS} it has been proved that, for any modular quantal frame $Q$, in particular any inverse quantal frame, each left $Q$-module admits at most one inner product that makes it a complete Hilbert $Q$-module. Hence, for modules on such well behaved quantales, having a complete Hilbert $Q$-module structure is a property rather than additional structure. For inverse quantal frames that property is the same as being a sheaf in the sense of section~\ref{sec:shmoritaiqf}:

\begin{theorem}[\cite{Resende2012}]
Let $Q$ be an inverse quantal frame, and $X$ a left $Q$-module. The following conditions are equivalent:
\begin{enumerate}
\item $X$ is a $Q$-sheaf.
\item $X$ is a complete Hilbert $Q$-module.
\end{enumerate}
Moreover, for any $Q$-sheaf $X$:
\begin{enumerate}
\setcounter{enumi}{2}
\item The local sections are the same as the Hilbert sections.
\item For all $x\in X$ the support and the inner product are related by the formula
\[
\spp(x) = \langle x,x\rangle\wedge e.
\]
\item Any downwards-closed set $\sections\subset X$ of regular sections such that $\V\sections=1$ is a Hilbert basis.
\end{enumerate}
\end{theorem}

\begin{remark}
Complete Hilbert modules provide a robust theory of sheaves whose applicability extends beyond quantales of \'etale groupoids. In particular, any Grothendieck topos is, up to equivalence, the category of sheaves on a Grothendieck quantale~\cite{Heymans}, which is a type of modular unital quantal frame of which inverse quantal frames are examples. Alternatively, staying closer to the representation of toposes by groupoids of~\cite{JT}, one may go beyond \'etale groupoids but restricting to open groupoids $G$, such as Lie groupoids, that are equipped with good pseudogroups of local bisections such that the (non-unital) quantale $\opens(G)$ has an appropriate embedding into an inverse quantal frame $Q$~\cite{Protin,QR2}. Then the classifying topos $\mathsf{B}G$ is equivalent to the full subcategory of $Q$-$\Sh$ whose objects have $\opens(G)$-valued inner products~\cite{QR2}.
\end{remark}

Let us further recall some facts about principal bundles from~\cite{QR}.

\begin{definition*}
Let $Q$ be an inverse quantal frame, and $X$ a $Q$-sheaf. A \emph{principal section} of $X$ is a local section $\gamma$ such that $\langle\gamma,\gamma\rangle\in Q_0$. The set of all the principal sections is denoted by $\psections_X$, and a \emph{principal basis} is a Hilbert basis $\sections\subset\psections_X$.
\end{definition*}

Principal bases enable us to restrict to pseudogroup-valued inner products:

\begin{proposition}[\cite{QR}]\label{prop:princbasis}
Let $Q$ be an inverse quantal frame, and $\sections$ a Hilbert basis of a $Q$-sheaf. The following conditions are equivalent:
\begin{enumerate}
\item $\sections$ is a principal basis.
\item $\langle\gamma,\gamma'\rangle\in Q_\ipi$ for all $\gamma,\gamma'\in\sections$.
\end{enumerate}
\end{proposition}

In addition, the importance of principal sections is that they cater for simple module-theoretic characterizations of sheaves that are also principal bundles:

\begin{theorem}[\cite{QR}]
Let $G$ be an \'etale groupoid, and $X$ a $G$-sheaf (equivalently, an $\opens(G)$-sheaf). The following conditions are equivalent:
\begin{enumerate}
\item $X$ is a principal $G$-bundle (over the ``locale of orbits'' $X/G$).
\item $X$ has a principal basis.
\item $\psections_X$ is a principal basis.
\item $\V\psections_X=1$ (i.e., $X$ is covered by principal sections).
\end{enumerate}
\end{theorem}

Based on this, Morita equivalence for inverse quantal frames can be characterized, much in the same way as for noncommutative rings, in terms of quantale bimodules, which are the unital bi-action objects in the category of sup-lattices.
The appropriate bimodules are those that define principal sheaves with respect to both actions --- and thus correspond to Hilsum--Skandalis maps that are invertible 1-morphisms in the bicategory of localic \'etale groupoids (\cf\ Remark~\ref{rem:HSmaps}). A summary of these requirements can be given as follows~(\cf\ \cite[\S5.4--5.5]{QR}):

\begin{definition*}
Let $Q$ and $R$ be inverse quantal frames. By a \emph{$Q$-$R$-bisheaf} will be meant a (unital) $Q$-$R$-bimodule $X$ whose left and right actions make it both a left $Q$-sheaf and a right $R$-sheaf. 
A \emph{local bisection} of $X$ is an element $\gamma\in X$ which is a local section with respect to both sheaf structures, and we denote the set of all the local bisections by $\bisections_X$. 
Equivalently, a $Q$-$R$-bimodule $X$ is a $Q$-$R$-bisheaf if it satisfies the following conditions:
\begin{description}
\item[Inner products] $X$ is equipped with two inner products
\[
\langle-,-\rangle: X\times X\to Q\quad\quad\text{and}\quad\quad [-,-]:X\times X\to R
\]
(with $\langle-,-\rangle$ being left $Q$-linear and $[-,-]$ being right $R$-linear).
\item[Bisections] There is $\sections\subset X$ such that for all $x\in X$
\[
\V_{\gamma\in\sections} \langle x,\gamma\rangle \gamma = \V_{\gamma\in\sections} \gamma[\gamma,x]=x\;.
\]
\end{description}
(Then $\bisections_X$ is the largest such set $\sections$.)
We shall usually denote the corresponding support operators by
\[
\spp :X\to Q_0\quad\quad\text{and}\quad\quad\tspp :X\to R_0\;,
\]
so for all $x\in X$ we have
\[
\spp(x) = \langle x,x\rangle\wedge e_Q\quad\quad\text{and}\quad\quad\tspp(x)=[x,x]\wedge e_R\;.
\]
Finally, a $Q$-$R$-bisheaf is \emph{biprincipal} if the following additional conditions are satisfied:
\begin{description}
\item[Covering conditions]
\[
\V_{\gamma\in\bisections_X} \langle \gamma,\gamma\rangle = e_Q\quad \textrm{and}\quad \V_{\gamma\in\bisections_X}[\gamma,\gamma]= e_R.
\]
\item[Inner product associativity] For all $\gamma,\gamma',\gamma''\in\bisections_X$ we have
\[
\langle \gamma,\gamma'\rangle \gamma''= \gamma[\gamma',\gamma''].
\]
\end{description}
\end{definition*}

\begin{theorem}[{\cite[\S5.5--5.6]{QR}}]\label{thm:QRbiprincipal}
Two inverse quantal frames $Q$ and $R$ are Morita equivalent if and only if there is a biprincipal $Q$-$R$-bisheaf.
\end{theorem}

\subsection{The main theorem}

Based on the characterization of Morita equivalence for inverse quantal frames via biprincipal bisheaves, in this section we prove Theorem~\ref{theoremB} of the Introduction.

\begin{lemma}\label{lem:biprincipal1}
Let $Q$ and $R$ be inverse quantal frames and $\Xi$ a biprincipal $Q$-$R$-bisheaf. Then $\bisections_\Xi$ is a $(Q_\ipi,R_\ipi)$-equivalence bimodule.
\end{lemma}
\begin{proof}
The covering conditions in the definition of biprincipal bisheaf imply that that the bisections are principal for both sheaf structures, so by Proposition~\ref{prop:princbasis} the inner products restricted to $\bisections_\Xi$ are valued in $Q_\ipi$ and $R_\ipi$, respectively:
\[
\langle-,-\rangle\colon \bisections_\Xi\times \bisections_\Xi\to Q_\ipi\quad\quad\text{and}\quad\quad[-,-]\colon \bisections_\Xi\times \bisections_\Xi\to R_\ipi\;.
\]
Then the covering conditions of the definition of pseudogroup bimodule follow from those of the definition of a biprincipal $Q$-$R$-bisheaf. Proving the remaining axioms for the inner products of a $(Q_\ipi,R_\ipi)$-biaction is straightforward, only the axiom $\langle \gamma,\gamma\rangle\gamma=\gamma$ requires some justification: the support $\spp\colon \Xi\to Q_0$ is defined by $\spp(x)=\langle x,x\rangle\wedge e_Q$ for all $x\in \Xi$, and thus for $\gamma\in\bisections_\Xi$ we have $\spp(\gamma)=\langle\gamma,\gamma\rangle$, which implies $\langle\gamma,\gamma\rangle\gamma=\spp(\gamma)\gamma=\gamma$. Finally, in order to conclude that $\bisections_\Xi$ is a $(Q_\ipi,R_\ipi)$-equivalence bimodule we only need to verify the completeness and distributivity conditions. The two distributivity conditions follow immediately from the distributivity of the actions and the inner products of sheaves on inverse quantal frames. In order to verify completeness, let $Z\subset\bisections_\Xi$ be a compatible set. This means it satisfies both left compatibility and right compatibility. Hence, $\V Z$ is both a local section for the $Q$-sheaf structure and a local section for the $R$-sheaf structure, which means that it is a local bisection, thus concluding the proof.
\end{proof}

Now for each $(S,T)$-bimodule $X$ we write $\lcc(X)$ for the sup-lattice whose elements are the downwards closed subsets of $X$ that are closed under the formation of joins of compatible sets. Notice that, despite the fact that the notation coincides with that of section~\ref{sec:equivcat}, here the meaning of $\lcc(X)$ is slightly different because the completion of the elements of $\lcc(X)$ is done with respect to compatibility rather than left compatibility --- \cf\ Remark~\ref{rem:collision}. The technical details pertaining to the universal properties of $\lcc(X)$ are very similar to those based on Proposition~\ref{prop:presentations} which were used in Lemma~\ref{lem:sheavesfrommodules}, so we will slightly abbreviate the presentation.

\begin{lemma}\label{lem:biprincipal2}
Let $S$ and $T$ be pseudogroups, and $X$ an equivalence $(S,T)$-bimodule. 
Then $\lcc(X)$ is a biprincipal $\lcc(S)$-$\lcc(T)$-bisheaf.
\end{lemma}
\begin{proof}
The general properties of module presentations allow us to extend the $(S,T)$-bimodule structure of $X$ to an $\lcc(S)$-$\lcc(T)$-bimodule structure on $\lcc(X)$, and checking the bi-action associativity properties is straightforward. Similarly, the inner products of $X$ extend to Hilbert module inner products on $\lcc(X)$. Again, showing that these are Hermitian and left/right-linear is straightforward. Let us just verify the left $\lcc(S)$-equivariance: let $I\in\lcc(S)$ and $J,J'\in\lcc(X)$; then we obtain
\[
\langle IJ,J'\rangle = \V_{a\in I}\V_{x\in J}\V_{y\in J'}\langle ax,y\rangle^{\downarrow} 
=\V_{a\in I}\V_{x\in J}\V_{y\in J'}a^{\downarrow}\langle x,y\rangle^{\downarrow}= I\langle J,J'\rangle\;.
\]
The associativity condition $\langle J,J'\rangle J''= J[J',J'']$ is proved in the same way, so let us prove that there is a common Hilbert basis for the two inner products.
Write $\sections\subset\lcc(X)$ for the set of principal order-ideals. Let $J\in\lcc(X)$ and $x\in X$ be such that $J\subset x^{\downarrow}$. Then $J$ has an upper bound $x$, which implies that it is a compatible set and therefore must contain its own join $y=\V J$. Hence $J=y^{\downarrow}$, and we conclude that $\sections$ is downwards closed in $\lcc(X)$. Moreover, $\V\sections=X$, so $\sections$ is a downwards closed cover of $\lcc(X)$. Furthermore, each $\gamma=x^{\downarrow}\in\sections$ is a Hilbert section with respect to both inner products, since for each $J\in\lcc(X)$ we have
\[
\langle J,\gamma\rangle\gamma =\V_{y\in J}\bigl(\langle y,x\rangle x\bigr)^{\downarrow}=\V_{y\in J} \bigl(y[x,x]\bigr)^{\downarrow}\subset \V_{y\in J}y^{\downarrow}=J\;,
\]
and it is a regular Hilbert section:
\[
\langle\gamma,\gamma\rangle\gamma = \bigl(\langle x,x\rangle x\bigr)^{\downarrow}=x^{\downarrow}=\gamma\;.
\]
From Proposition~\ref{parseval} it follows that $\sections$ is a Hilbert basis for the $\lcc(S)$-module structure, so $\lcc(X)$ is an $\lcc(S)$-sheaf. Analogously we see that it is a right $\lcc(T)$-sheaf. Finally, we have
\[
\V_{\gamma\in\sections} \langle\gamma,\gamma\rangle = \V_{x\in X}\langle x,x\rangle^{\downarrow} = (e_S)^{\downarrow}=e_{\lcc(S)}\;,
\]
and, similarly,
\[
\V_{\gamma\in\sections} [\gamma,\gamma] = e_{\lcc(T)}\;,
\]
so $\lcc(X)$ is a biprincipal $\lcc(S)$-$\lcc(T)$-bisheaf.
\end{proof}

Finally we arrive at Theorem~\ref{theoremB}:

\begin{theorem}\label{them:three} Let $S$ and $T$ be pseudogroups.
Then $S$ and $T$ are Morita equivalent if and only if there is an $(S,T)$-equivalence bimodule.
\end{theorem}

\begin{proof}
Assume that $S$ and $T$ are Morita equivalent. Then by Theorem~\ref{theoremA} the inverse quantal frames $\lcc(S)$ and $\lcc(T)$ are Morita equivalent, so there is a biprincipal $\lcc(S)$-$\lcc(T)$-bisheaf $\Xi$ due to Theorem~\ref{thm:QRbiprincipal}. Then Lemma~\ref{lem:biprincipal1} gives us an $(\lcc(S)_\ipi,\lcc(T)_\ipi)$-equivalence bimodule $\bisections_\Xi$, which in turn can be made an $(S,T)$-equivalence bimodule due to the pseudogroup isomorphisms $S\cong\lcc(S)_\ipi$ and $T\cong\lcc(T)_\ipi$.

Now assume that $X$ is an $(S,T)$-equivalence bimodule. Then Lemma~\ref{lem:biprincipal2} gives us a biprincipal $\lcc(S)$-$\lcc(T)$-bisheaf $\lcc(X)$, hence showing that $\lcc(S)$ and $\lcc(T)$ are Morita equivalent due to Theorem~\ref{thm:QRbiprincipal}. Then $S$ and $T$ are Morita equivalent, by Theorem~\ref{theoremA}.
\end{proof}

Although irrelevant for the main conclusion that was the goal of this section, we remark that the above results can be strengthened by establishing isomorphisms $X\cong\sections_{\lcc(X)}$ and $\lcc(\sections_\Xi)\cong\Xi$,  indeed showing that there is an equivalence, analogous to the equivalence of categories of Theorem~\ref{them:one}, between appropriate categories of $(S,T)$-equivalence bimodules and of biprincipal $\lcc(S)$-$\lcc(T)$-bisheaves. This, in turn, can be placed in a wider bicategorical context (\cf\ \cite{Resende2015}).

\section{Enlargements}

In this section, we shall describe a purely algebraic characterization of when two pseudogroups are Morita equivalent.
It relies on the concept of an equivalence bimodule introduced in the previous section.
 
Let $T$ be a pseudogroup and $e$ an idempotent in $T$.
Then $eTe$ is a pseudogroup in its own right with respect to the inherited operations of multiplication, inversion
and compatible joins that we call a {\em local pseudogroup} of $T$.
It is not a `subpseudogroup' because the identity of $eTe$ is $e$ rather than the identity of $T$.
We say that $T$ is a {\em sup-enlargement} of $S = eTe$ if the following two conditions hold:
\begin{description}

\item[{\rm (E1)}] $S = STS$.

\item[{\rm (E2)}] $T = (TST)^{\V}$.

\end{description}
Enlargments were first defined in \cite{Lawson1996} and the above definition was motivated by the additive enlargements defined in \cite{Wehrung}.
The following is actually just a consequence of (E1); 
see \cite{Lawson1996}.

\begin{lemma}\label{lem:arret} Let $T$ be a sup-enlargement of $S = eTe$.
\begin{enumerate}
\item $S$ is an order-ideal of $T$.
\item If $t \in T$ is such that $t^{-1}t,tt^{-1} \in S$ then $t \in S$.
\end{enumerate}
\end{lemma}

Let $S$ and $T$ be pseudogroups.
A {\em joint sup-enlargement} of $S$ and $T$ is a pseudogroup $U$ which is a sup-enlargement of both $S$ and $T$.
The following is the first part of Theorem~\ref{theoremC} from the Introduction.

\begin{theorem}\label{them:enlargement} Let the pseudogroups $S$ and $T$ have a joint sup-enlargment $U$.
Then there is an $(S,T)$-equivalence bimodule.
\end{theorem}
\begin{proof} Denote the identities of $S$ and $T$ by $e$ and $f$, respectively.
Put $X = \mathsf{E}(S)U\mathsf{E}(T) = eUf$.
Define the action $S \times X \rightarrow X$ by left multiplication and $X \times T \rightarrow X$ by right multiplication.
Define $\langle -,- \rangle \colon X \times X \rightarrow S$ by $\langle x,y \rangle = xy^{-1}$ and 
$[-,-] \colon X \times X \rightarrow T$ by $[x,y] = x^{-1}y$.
Then it is a routine, if somewhat laborious, business to check that with these definitions,
we have defined an equivalence bimodule and so by Theorem~\ref{them:three} we have proved that the pseudogroups $S$ and $T$ are Morita equivalent.
Observe that conditions (1)---(7) in the definition of a biaction are routine to check;
the covering conditions follow from condition (E2) in the definition of an enlargement;
the conditions in the definition of a bimodule follow from the fact that $U$ is a pseudogroup.
\end{proof}

\begin{example}
Denote the finite symmetric inverse monoid on the letters $\{1, \ldots, n\}$ by $I_{n}$.
These are all pseudogroups and they are all Morita equivalent as pseudogroups,
as we now show.
Denote the unique element of $I_{n}$ that maps $\{x\}$ to $\{y\}$ by $f_{y,x}$.
Observe that $f_{1,1}I_{n}f_{1,1} \cong I_{1}$,
and that for any $1 \leq x,y \leq n$, we have that $f_{y,x} = f_{y,1}f_{1,1}f_{1,x}$.
Thus $I_{n} = \left(I_{n}f_{1,1} I_{n} \right)^{\V}$.
It follows that $I_{1}$ is Morita equivalent to $I_{n}$ for all finite $n$ as a pseudogroup.
Thus, since Morita equivalence is an equivalence relation, all finite symmetric inverse monoids are Morita equivalent {\em as pseudogroups}.
However, by \cite[Corollary 5.2]{Steinberg}, Morita equivalent inverse semigroups have the same number of $\mathscr{D}$-classes.
It follows that two finite symmetric inverse monoids $I_{m}$ and $I_{n}$ are Morita equivalent as inverse semigroups
if and only if $m = n$.
\end{example}

The proof of the following lemma is routine.

\begin{lemma}\label{lem:dual} Let $(S,T,X,\langle -,- \rangle,[-,-])$ be an equivalence bimodule.
Define $\overline{X} = \{\bar{x} \st x \in X\}$.
Define a right action of $S$ on $\overline{X}$ by $\bar{x}s = \overline{s^{-1}x}$ and define a left action of
$T$ on $\overline{X}$ by $t\bar{x} = \overline{xt^{-1}}$.
Then $(T,S,\overline{X},[-,-], \langle -,- \rangle)$  is an equivalence bimodule.
We call this the {\em dual} of  $(S,T,X,\langle -,- \rangle,[-,-])$.
\end{lemma}

The next lemma will be useful in the proof of our main theorem.

The following is an adaptation of the main construction of \cite{BGR} to this setting.
The generalized rook conditions are a slight extension of what is to be found in \cite{Hines, Wallis, KLLR}.
It proves the second part of Theorem~\ref{theoremC} in the Introduction.

\begin{theorem}\label{them:rieffel} Let $S$ and $T$ be pseudogroups with an $(S,T)$-equivalence bimodule.
Then $S$ and $T$ have a joint sup-enlargement.
\end{theorem}
\begin{proof} We shall construct a pseudogroup $U$ that contains copies of $S$ and $T$ and is a sup-enlargement
of both.
This is a sequence of laborious verifications.
We simply give the main stages here.
Denote the identity of $S$ by $e_{S}$ and the identity of $T$ by $e_{T}$.
Let 
$$(S,T,X,\langle -,- \rangle,[-,-])$$ 
be an $(S,T)$-equivalence bimodule.
By Lemma~\ref{lem:dual}, we can also construct the dual equivalence bimodule
$$(T,S,\overline{X},[-,-], \langle -,- \rangle).$$
Denote by
$$U = \left(\begin{array}{cc} S & X \\ \overline{X} & T \end{array} \right)$$
the set of all $2 \times 2$ matrices
$$\left(\begin{array}{cc} s & x \\ \bar{y} & t \end{array} \right)$$
satisfying what we shall term the {\em generalized rook conditions}:
\begin{itemize}
\item $s^{-1}x = 0$.
\item $yt = 0$. 
\item $sy = 0$. 
\item $xt^{-1} = 0$.
\end{itemize}
Define a product in $U$ by
$$
\left(\begin{array}{cc} s_{1} & x_{1} \\ \overline{y_{1}} & t_{1} \end{array} \right)
\left(\begin{array}{cc} s_{2} & x_{2} \\ \overline{y_{2}} & t_{2} \end{array} \right)
=
\left(\begin{array}{cc} s_{1}s_{2} \vee \langle x_{1}, y_{2} \rangle & s_{1}x_{2} \vee x_{1}t_{2} 
\\ \overline{s_{2}^{-1}y_{1} \vee y_{2}t_{1}^{-1}} & [y_{1},x_{2}] \vee t_{1}t_{2} \end{array} \right).$$
This is really a matrix product in which $\vee$ replaces addition;
multiplication is treated in various ways:
as semigroup multiplication,
as a semigroup action,
and we treat a product of $x$ and $\bar{y}$ as $\langle x, y \rangle$
and we treat a product of $\bar{x}$ and $y$ as $[x,y]$.\\

\noindent
(1) {\em The product is well-defined.} 
We need to check that the joins as written are defined:
\begin{itemize}

\item The fact that $s_{1}s_{2} \sim \langle x_{1}, y_{2} \rangle$ follows from the fact that $s_{1}^{-1}x_{1} = 0$ and $s_{2}y_{2} = 0$.

\item To prove that $s_{1}x_{2} \sim x_{1}t_{2}$ we show that $\langle s_{1}x_{2}, x_{1}t_{2} \rangle = 0$
and $[s_{1}x_{2}, x_{1}t_{2}] = 0$ using Lemma~\ref{surjectivia}.
To do this, requires us to use $x_{2}t_{2}^{-1} = 0$ and $s_{1}^{-1}x_{1} = 0$
and Proposition~\ref{lem:biaction-action}.

\item To prove that $s_{2}^{-1}y_{1} \sim y_{2}t_{1}^{-1}$ is similar to the above proof using the fact that $y_{1}t_{1} = 0$ and $s_{2}y_{2} = 0$.

\item The fact that $[y_{1},x_{2}] \sim t_{1}t_{2}$ follows from the fact that $y_{1}t_{1} = 0$ and $x_{2}t_{2}^{-1} = 0$.

\end{itemize}
Thus all the joins that are claimed to exist, do really exist.
It is routine to check that the generalized rook conditions hold for the new matrix;
in checking these, you will find it useful to observe that if $ ax = 0$ and $b \leq a$ then $bx = 0$.\\

\noindent
(2) {\em $U$ is a monoid with zero.}
We check that the multiplication is associative.
We therefore have to calculate the following product in two different ways and show that they are equal:
$$\left(\begin{array}{cc} s_{1} & x_{1} \\ \overline{y_{1}} & t_{1} \end{array} \right)
\left(\begin{array}{cc} s_{2} & x_{2} \\ \overline{y_{2}} & t_{2} \end{array} \right)
\left(\begin{array}{cc} s_{3} & x_{3} \\ \overline{y_{3}} & t_{3} \end{array} \right).$$
When we multiply the first two matrices together first,  we get the following four elements:
\begin{enumerate}
\item $s_{1}s_{2}s_{3} \vee \langle x_{1}, y_{2}\rangle s_{3} \vee \langle s_{1}x_{2} \vee x_{1}t_{2}, y_{3} \rangle$. 
We have used the fact that multiplication distributes over joins.
We can expand the final term to get
$s_{1}s_{2}s_{3} \vee \langle x_{1}, y_{2}\rangle s_{3} 
\vee
s_{1}\langle x_{2}, y_{3} \rangle
\vee
\langle x_{1}t_{2}, y_{3} \rangle$.

\item $s_{1}s_{2}x_{3} \vee \langle x_{1}, y_{2}\rangle x_{3} \vee s_{1}x_{2}t_{3} \vee x_{1}t_{2}t_{3}$.
We have used the fact that the action distributes over joins.
\item $s_{3}^{-1}s_{2}^{-1}y_{1} \vee s_{3}^{-1}y_{2}t_{1}^{-1} \vee y_{3}[x_{2}, y_{1}] \vee y_{3}t_{2}^{-1}t_{1}^{-1}$.

\item $[s_{2}^{-1}y_{1} \vee y_{2}t_{1}^{-1}, x_{3}] \vee [y_{1}, x_{2}]t_{3} \vee t_{1}t_{2}t_{3}$.
This expands to the following
$[s_{2}^{-1}y_{1}, x_{3}] \vee [y_{2}t_{1}^{-1}, x_{3}] \vee [y_{1}, x_{2}]t_{3} \vee t_{1}t_{2}t_{3}$.

\end{enumerate}
When we multiply the second two matrices together first, we get the following four elements:
\begin{enumerate}
\item $s_{1}s_{2}s_{3} \vee s_{1}\langle x_{2}, y_{3} \rangle \vee \langle x_{1}, s_{3}^{-1}y_{2} \vee y_{3}t_{2}^{-1}\rangle$.
We can expand the final term to get
$s_{1}s_{2}s_{3} \vee s_{1}\langle x_{2}, y_{3} \rangle \vee 
\langle x_{1}, s_{3}^{-1}y_{2} \rangle
\vee
\langle x_{1}, y_{3}t_{2}^{-1}\rangle$.
By \cite[Prop.~2.3 part (3)]{Steinberg},
we have that $\langle x_{1}, s_{3}^{-1}y_{2} \rangle = \langle x_{1}, y_{2} \rangle s_{3}$
and by \cite[Prop.~2.3 part (9)]{Steinberg},
we have that $\langle x_{1}, y_{3}t_{2}^{-1} \rangle = \langle x_{1}t_{2}, y_{3}\rangle$.

\item $s_{1}s_{2}x_{3} \vee s_{1}x_{2}t_{3} \vee x_{1}[y_{2},x_{3}] \vee x_{1}t_{2}t_{3}$.
This is equal to 
$s_{1}s_{2}x_{3} \vee s_{1}x_{2}t_{3} \vee \langle x_{1}, y_{2} \rangle x_{3} \vee x_{1}t_{2}t_{3}$.

\item $s_{3}^{-1}s_{2}^{-1}y_{1} \vee \langle y_{3}, x_{2} \rangle y_{1} \vee s_{3}^{-1}y_{2}t_{1}^{-1} \vee y_{3}t_{2}^{-1}t_{1}^{-1}$.
This is equal to 
$s_{3}^{-1}s_{2}^{-1}y_{1} \vee y_{3}[x_{2}, y_{1}] \vee s_{3}^{-1}y_{2}t_{1}^{-1} \vee y_{3}t_{2}^{-1}t_{1}^{-1}$.

\item $[y_{1}, s_{2}x_{3} \vee x_{2}t_{3}] \vee t_{1}[y_{2},x_{3}] \vee t_{1}t_{2}t_{3}$.
This expands to
$[y_{1}, s_{2}x_{3}] \vee [y_{1}, x_{2}t_{3}] \vee t_{1}[y_{2}, y_{3}] \vee t_{1}t_{2}t_{3}$.
\end{enumerate}
We need to show that corresponding entries (these have the same number labelling) are equal.
These all check out, often using Proposition~\ref{lem:biaction-action}.
It is easy to verify that
the identity is
$$\left(\begin{array}{cc} e_{S} & 0 \\ \bar{0} & e_{T} \end{array} \right)$$
and the zero is
$$\left(\begin{array}{cc} 0 & 0 \\ \bar{0} & 0 \end{array} \right).$$

\noindent
(3) {\em $U$ is an inverse monoid.}
Define
$$\left(\begin{array}{cc} s & x \\ \bar{y} & t \end{array} \right)^{-1}
=
\left(\begin{array}{cc} s^{-1} & y \\ \bar{x} & t^{-1} \end{array} \right).$$
It can be checked that this is a well-defined operation meaning that
the generalized rook conditions hold.
Then
$$\left(\begin{array}{cc} s & x \\ \bar{y} & t \end{array} \right)
=
\left(\begin{array}{cc} s & x \\ \bar{y} & t \end{array} \right)
\left(\begin{array}{cc} s^{-1} & y \\ \bar{x} & t^{-1} \end{array} \right)
\left(\begin{array}{cc} s & x \\ \bar{y} & t \end{array} \right)$$
and
$$
\left(\begin{array}{cc} s & x \\ \bar{y} & t \end{array} \right)
\left(\begin{array}{cc} s^{-1} & y \\ \bar{x} & t^{-1} \end{array} \right)
=
\left(\begin{array}{cc} ss^{-1} \vee \langle x,x \rangle & 0 \\ \bar{0} & [y,y] \vee tt^{-1} \end{array} \right)$$
and
$$
\left(\begin{array}{cc} s^{-1} & y \\ \bar{x} & t^{-1} \end{array} \right)
\left(\begin{array}{cc} s & x \\ \bar{y} & t \end{array} \right)
=
\left(\begin{array}{cc} s^{-1}s \vee \langle y,y \rangle & 0 \\ \bar{0} & [x,x] \vee t^{-1}t \end{array} \right).
$$
The idempotents of $U$ can be characterized as precisely those elements of the form
$$\left(\begin{array}{cc} e & 0 \\ \bar{0} & f \end{array} \right)$$
where $e \in \mathsf{E}(S)$ and $f \in \mathsf{E}(T)$.
It follows that the idempotents commute and so $U$ is an inverse monoid with zero.\\

\noindent
(3) {\em $U$ contains copies of $S$ and $T$. }
Define
$$\mathbf{e}_{S} = \left(\begin{array}{cc} e_{S} & 0 \\ \bar{0} & 0 \end{array} \right)
\text{ and }
\mathbf{e}_{T} = \left(\begin{array}{cc} 0 & 0 \\ \bar{0} & e_{T} \end{array} \right).$$
Then
$$S \cong \mathbf{e}_{S}U\mathbf{e}_{S} = S'
\text{ and }
T \cong \mathbf{e}_{T}U \mathbf{e}_{T} = T'.$$
It follows that $U$ contains copies of $S$ and $T$.\\

\noindent
(4) {\em $U$ is a pseudogroup.} If
$$
\left(\begin{array}{cc} s_{1} & x_{1} \\ \bar{y}_{1} & t_{1} \end{array} \right)
\sim
\left(\begin{array}{cc} s_{2} & x_{2} \\ \bar{y}_{2} & t_{2} \end{array} \right)$$
then
$s_{1} \sim s_{2}$, $x_{1} \sim x_{2}$, $y_{1} \sim y_{2}$ and $t_{1} \sim t_{2}$.
It can now be checked that
$$
\left(\begin{array}{cc} s_{1} & x_{1} \\ \bar{y}_{1} & t_{1} \end{array} \right)
\vee
\left(\begin{array}{cc} s_{2} & x_{2} \\ \bar{y}_{2} & t_{2} \end{array} \right)$$
exists and is equal to
$$
\left(\begin{array}{cc} s_{1} \vee s_{2} & x_{1} \vee x_{2} \\ \overline{y_{1} \vee y_{2}} & t_{1} \vee t_{2} \end{array} \right).$$
It is routine to check that this is a well-defined operation.
We now prove that $U$ is a pseudogroup.
By symmetry, it is enough to check that multiplication distributes from the left over arbitrary joins.
We calculate
$$\left(\begin{array}{cc} s & x \\ \bar{y} & t \end{array} \right)
\left(\bigvee_{i \in I}
\left(\begin{array}{cc} s_{i} & x_{i} \\ \bar{y_{i}} & t_{i} \end{array} \right)
\right).
$$
This is equal to
$$\left(\begin{array}{cc} s & x \\ \bar{y} & t \end{array} \right)
\left(\begin{array}{cc} \bigvee_{i \in I} s_{i} & \bigvee_{i \in I} x_{i} \\ \overline{(\bigvee_{i \in I} y_{i})} & \bigvee_{i \in I} t_{i} \end{array} \right).
$$
It is now routine to check the answer.\\

\noindent
(5) {\em $U$ is an enlargement of the copy of $S$.}
It remains to show that condition (E2) in the definition of an enlargement holds.
We shall do this with respect to $S$, since by symmetry the result will then be true for $T$.
We shall use the copy $S'$ of $S$.
Let 
$$\left(\begin{array}{cc} s & x \\ \bar{y} & t \end{array} \right)$$
be an arbitrary element of $U$.
Then
$$\left(\begin{array}{cc} s & x \\ \bar{y} & t \end{array} \right)
=
\left(\begin{array}{cc} \V_{i \in I} \langle x_{i}, y_{i} \rangle & x \\ \bar{y} & \V_{j \in J} [u_{j},v_{j}] \end{array} \right)$$
where we have used Lemma~\ref{lem:coveringconditions}.
Consider now the following four kinds of element of $U$:
\begin{enumerate}
\item $\left(\begin{array}{cc} \langle x_{i}, y_{i} \rangle & 0 \\ \bar{0} & 0 \end{array} \right)$.
\item $\left(\begin{array}{cc} 0 & 0 \\ \bar{y} & 0 \end{array} \right) 
= \left(\begin{array}{cc} 0 & 0 \\ \bar{y} & 0 \end{array} \right)\left(\begin{array}{cc} \langle y,y \rangle & 0 \\ \bar{0} & 0 \end{array} \right)$.
\item $\left(\begin{array}{cc}  0 & x \\ \bar{0} & 0 \end{array} \right)
= \left(\begin{array}{cc} \langle x, x \rangle & 0 \\ \bar{0} & 0 \end{array} \right)\left(\begin{array}{cc} 0  & x \\ \bar{0} & 0 \end{array} \right)$.
\item $\left(\begin{array}{cc} 0 & 0 \\ \bar{0} & [u_{j},v_{j}] \end{array} \right)
= 
\left(\begin{array}{cc} 0 & 0 \\ \bar{u}_{j} & 0 \end{array} \right)
\left(\begin{array}{cc} \langle u_{j}, u_{j} \rangle  & 0 \\ \bar{0} & 0 \end{array} \right)
\left(\begin{array}{cc} 0  & v_{j} \\ \bar{0} & 0 \end{array} \right)$.
\end{enumerate}
It follows that each of these four kinds of elements belongs to $US'U$ and so each element of $U$ is a join of elements of $US'U$.
\end{proof}

\section{Morita invariant properties}

A property $\mathscr{P}$ of pseudogroups is said to be {\em Morita invariant} if whenever a pseudogroup $S$ has the property $\mathscr{P}$
and $T$ is Morita equivalent to $S$ then $T$ has the property $\mathscr{P}$.
In this section, we describe two Morita invariant properties.

Let $S$ be a pseudogroup.
A {\em sup-ideal} of $S$ is an ideal $I \subseteq S$ which is closed under arbitrary compatible joins.

\begin{lemma}\label{lem:ideal} Let $S$ be a pseudogroup.
Let $A$ be an ideal.
Then $A^{\V}$ is a sup-ideal.
\end{lemma}
\begin{proof} The set $A^{\V}$ is sup-closed.
It remains only to check that it is an ideal.
But this follows from the fact that products distribute over arbitrary compatible joins.
\end{proof}

The following definition is motivated by \cite{K}.
A pseudogroup $S$ is said to be {\em $0$-simplifying} if there are no non-trivial sup-ideals.
Let $e$ and $f$ be idempotents.
A subset $X \subseteq S$ is called a {\em pencil} from $e$ to $f$ if 
$e = \V_{x \in X} \mathbf{d}(x)$ and $f \geq \V_{x \in X} \mathbf{r}(x)$.
Define $e \preceq f$ if there is a pencil from $e$ to $f$.
The proof of the following is straightforward.

\begin{lemma}\label{lem:pencil} 
The relation $\preceq$ is a preorder on the set $\mathsf{E}(S)$.
\end{lemma}

Denote by $\equiv$ the equivalence relation induced on $\mathsf{E}(S)$ by $\preceq$.
Compare the following with \cite[Corollary, page 150]{K}; despite appearances, there is no number missing following the word `Corollary'
in this reference.

\begin{lemma}\label{lem:simple} Let $S$ be a pseudogroup.
The following are equivalent.
\begin{enumerate}
\item $S$ is $0$-simplifying.
\item $\equiv$ is the universal relation on the set of non-zero idempotents.
\item For any two non-zero idempotents $e$ and $f$ 
there is a subset $Z \subseteq S$ such that $e = \V_{z \in Z} \mathbf{d}(z)$ and $f = \V_{z \in Z} \mathbf{r}(z)$.
\end{enumerate}
\end{lemma}
\begin{proof} (1)$\Rightarrow$(2).
Let $e$ and $f$ be non-zero idempotents.
Then by assumption and Lemma~\ref{lem:ideal},
we have that
$S = (SeS)^{\V} = (SfS)^{\V}$.
Now, $e \in (SfS)^{\V}$.
Thus $e = \bigvee_{i \in I} e_{i}$ where $e_{i} \in SfS$.
We can write $e_{i} = e_{i}a_{i}fb_{i}e_{i}$ where $a_{i}, b_{i} \in S$.
Put $x_{i}^{-1} = e_{i}a_{i}f$ and $y_{i} = fb_{i}e_{i}$.
It is easy to check that $y_{i} = x_{i}^{-1}$ (inverses are unique in inverse semigroups).
Thus  $e = \V_{i \in I} x_{i}^{-1}x_{i}$ where
$\mathbf{r}(x_{i}) \leq f$.
It follows that $e \preceq f$.
Symmetry delivers the result.

(2)$\Rightarrow$(3). Let $X$ be a pencil from $e$ to $f$ and let $Y$ be a pencil from $f$ to $e$.
Put $Z = X \cup Y^{-1}$.
Then 
$$\V_{z \in Z} \mathbf{d}(z) = \left( \V_{x \in X} \mathbf{d}(x) \right) \vee \left( \V_{y \in Y} \mathbf{r}(y) \right) = e$$
and, by symmetry,
$$\V_{z \in Z} \mathbf{r}(z) = f.$$

(3)$\Rightarrow$(1). Let $I \neq \{0\}$ be a sup-ideal of $S$.
Since there exists $a \in I$ with $a \neq 0$, it follows that there exists an idempotent $e \in I$ where $e \neq 0$.
Let $f \in S$ be any non-zero idempotent.
Then by (3), we can find a subset $Z \subseteq S$ such that $e = \V_{z \in Z} \mathbf{d}(z)$ and $f = \V_{z \in Z} \mathbf{r}(z)$.
Every ideal is an order-ideal so that $\mathbf{d}(z) \in I$ for all $z \in Z$.
But $z = z \mathbf{d}(z)$ and so $z \in I$.
Thus we have shown that $Z \subseteq I$.
But $I$ is $\V$-closed and so $f \in I$.
It follows that $\mathsf{E}(S) \subseteq I$.
But if $s \in S$ is arbitrary, then $s = s(s^{-1}s)$.
It follows that $s \in I$ and so we have proved that $S = I$, as required. 
\end{proof}

The following lemma was actually proved in \cite{Lawson1996} but is so useful in
our calculation that we reprove it here.

\begin{lemma}\label{lem:d-relation} Let $S$ and $T$ be pseudogroups such that $T$ is a sup-enlargement of $S$.
Then each idempotent in $TST$ is $\mathscr{D}$-related to an idempotent in $S$.
\end{lemma}
\begin{proof} Let $e \in TST$ be an idempotent.
Then $e = (ets)(s^{-1}st'e)$ where $s \in S$ and $t \in T$.
Put $x = ets$ and $y = s^{-1}st'e$.
Then $xyx = ex = x$ and $yxy = ye = y$.
It follows that $x$ and $y$ are inverses of each other.
Thus $e = xx^{-1}$ and $f = x^{-1}x \in STS = S$. 
\end{proof}

We can now state our first Morita invariant property.

\begin{theorem}\label{prop:simplifying-mi} Let $S$ and $T$ be Morita equivalent pseudogroups.
Then $S$ is $0$-simplifying if and only if $T$ is $0$-simplifying.
\end{theorem}
\begin{proof} It is enough to prove the result under the assumption that $T$ is a sup-enlargement of $S$.
Because $S$ is an order-ideal of $T$ by Lemma~\ref{lem:arret}, it is straightforward to check that if $T$ is $0$-simplifying then $S$ is $0$-simplifying.
We prove the converse.
Let $S$ be $0$-simplifying.
We prove that $T$ is $0$-simplifying.
We begin with a special case where $e$ and $f$ are idempotents in $TST$.
We use Lemma~\ref{lem:d-relation}.
If $e \in \mathsf{E}(TST)$ then $e \stackrel{a}{\rightarrow} e'$
for some $a \in T$ where $e' \in \mathsf{E}(S)$,
and that if $f\in \mathsf{E}(TST)$ then  $f\stackrel b\rightarrow f’$ for some $b\in T$ where $f’\in E(S)$. 
By assumption, $e’\preceq f’$
and so there is a pencil $X$ from $e'$ to $f'$.
It can now be checked that $b^{-1}Xa$ is a pencil from $e$ to $f$.
Thus $e \preceq f$.
Now let $e$ and $f$ be arbitrary idempotents in $T$.
The $e = \bigvee_{i \in I} e_{i}$ and $f = \bigvee_{j \in J} f_{j}$
where $e_{i}$ and $f_{j}$ are idempotents in $TST$.
Then $e_{i} \preceq f_{j}$, by our above result.
Let $X_{ij}$ be a pencil from $e_{i}$ to $f_{j}$.
Put $X = \bigcup_{i \in I, j \in J} X_{ij}$.
Then $X$ is a pencil from $e$ to $f$ and so $e \preceq f$.
By symmetry, $f \preceq e$.
It follows that $T$ is $0$-simplifying.
\end{proof}

Recall that an inverse semigroup $S$ is said to be {\em fundamental} if the only elements that commute with
all the idempotents are the idempotents themselves.

\begin{lemma}\label{lem:fundamental} 
Suppose that $T$ is a sup-enlargement of $S$.
Then $S$ is fundamental if and only if $T$ is fundamental.
\end{lemma}
\begin{proof} Suppose that $T$ is fundamental.
Then $S$ is isomorphic to a local submonoid of $T$.
It is well-known (and easy to prove) that $S$ must itself be fundamental.
We therefore need only prove the converse.
Suppose that $S$ is fundamental;
we prove that $T$ is fundamental.

We prove first that $TST$ is fundamental.
Let $a \in TST$ be an element that commutes with every idempotent in $TST$.
We prove that $a$ is itself an idempotent.
Observe first that $\mathbf{d}(a) = \mathbf{r}(a) = e$ where $e \in TST$.
By Lemma~\ref{lem:d-relation}, 
there exists an idempotent $f \in S$ such that  $e \mathscr{D} f$.
Let $e \stackrel{u}{\rightarrow} f$.
Then $s = uau^{-1} \in S$.
Let $e$ be any idempotent in $S$.
Then 
$$se = uau^{-1}e = ua (u^{-1}eu)u^{-1} = u (u^{-1}eu)au^{-1} = euau^{-1}  = es.$$
It follows that $s$ is an idempotent
and so $a$ is an idempotent.

We can now prove that $T$ is fundamental.
Let $a$ be an element of $T$ that commutes with all idempotents of $T$.
We prove that $a$ is an idempotent.
By definition, $a = \V_{i \in I} t_{i}$ where $t_{i} \in TST$.
We shall prove that each $t_{i}$ is an idempotent from which it follows that $a$ is an idempotent.
Let $i \in I$.
Then $t_{i} \leq a$ and so $t_{i} = a \mathbf{d}(t_{i})$.
Let $e \in \mathsf{E}(TST)$.
Then $t_{i}e = a \mathbf{d}(t_{i})e = ae \mathbf{d}(t_{i}) = ea\mathbf{d}(t_{i}) = e t_{i}$.
But $TST$ is fundamental and so 
it follows that $t_{i}$ is an idempotent and so $a$ is an idempotent.
\end{proof}

Our second Morita invariant property is now immediate by Lemma~\ref{lem:fundamental} and Theorem~\ref{them:rieffel}.

\begin{corollary}\label{prop:fundamental} Let $S$ and $T$ be Morita equivalent pseudogroups.
Then $S$ is fundamental if and only if $T$ is fundamental.
\end{corollary}

\appendix

\section*{Appendix}

\begin{proof}[Proof of Proposition~\ref{prop:presentations}]
$\mathcal P(X)$ is a left $\mathcal P(M)$-module under the pointwise action.
For all $A\subset M$ and $Y,Z\subset X$ the two residuations for the action are defined as follows:
\begin{eqnarray*}
A\backslash Z &=& \bigcup\bigl\{ W\in\mathcal P(X) : A W\subset Z\bigr\}
=\{x\in X : A x\subset Z\},\\
Z/Y &=& \bigcup\bigl\{ B\in\mathcal P(M) : B Y\subset Z\bigr\}
=\{a\in M : a Y\subset Z\}.
\end{eqnarray*}
Take for instance $Z/Y$. This coincides with the intersection
\[
\bigcap_{y\in Y}\{a\in M : ay\in Z\}.
\]
If $Z\in\langle X\vert R_X\rangle$ and $(S,T)\in R_M$ then, using the third joint stability condition,
\[
S\subset\{a\in M : ay\in Z\}\iff Sy\subset Z\iff Ty\subset Z\iff T\subset\{a\in M : ay\in Z\},
\]
and thus
$Z/Y\in \langle M\vert R_M\rangle$. Analogously we conclude that if $Z\in\langle X\vert R_X\rangle$ then $A\backslash Z\in\langle X\vert R_X\rangle$ for all $A\subset M$. Similar conclusions apply to the two residuations for the multiplication of $Q$.

Now let $j$ and $k$ be the closure operators on $\mathcal P(M)$ and $\mathcal P(X)$ which are determined by the intersection-closed subsets $\langle M\vert R_M\rangle$ and $\langle X\vert R_X\rangle$. The former is the set of fixed points of $j$, the latter is the set of fixed points of $k$, and we have two surjective sup-lattice homomorphisms
\[
j:\mathcal P(M)\to\langle M\vert R_M\rangle\quad\textrm{ and }\quad k:\mathcal P(X)\to\langle X\vert R_X\rangle.
\]
Let $A\subset M$ and $Y\subset X$. Then
\begin{eqnarray*}
AY\subset k(AY) &\iff& A\subset k(AY)/Y\iff j(A)\subset k(AY)/Y\\
&\iff& j(A)Y\subset k(AY)
\iff Y\subset j(A)\backslash k(AY)\\
&\iff& k(Y)\subset j(A)\backslash k(AY)\iff
j(A)k(Y)\subset j(AY),
\end{eqnarray*}
so the condition $j(A)k(Y)\subset k(AY)$ holds. Similarly, $j(A)j(B)\subset j(AB)$  for all $A,B\subset M$ ($j$ is a \emph{quantic nucleus}~\cite[\S 3.1]{Rosenthal}), and $j$ commutes with the involution. Hence, $\langle M\vert R_M\rangle$ has a well defined multiplication given by
\[
(A,B)\mapsto A\then B := j(AB),
\]
and, similarly, it has a well defined action on $\langle X\vert R_X\rangle$ given by
\[
(A,Z)\mapsto A\then Z :=k(AZ).
\]
So $j$ and $k$ preserve all the operations of the quantale and module structures,
\[
j(AB) = j(A)\then j(B)\quad\textrm{ and }\quad k(AZ) = j(A)\then k(Z),
\]
and, due to surjectivity, the required algebraic properties such as the associativity of the action of $\langle M\vert R_M\rangle$ on $\langle X\vert R_X\rangle$ are inherited from the corresponding properties of $\mathcal P(M)$ and $\mathcal P(X)$.

Now let $\Xi$ be a left $\langle M\vert R_M\rangle$-module, and $f:X\to\Xi$ an $M$-equivariant map that respects $R_X$. Due to the sup-lattice universal property of $\langle X\vert R_X\rangle$ there is a unique sup-lattice homomorphism $f^\sharp:\langle M\vert R_X\rangle\to\Xi$ such that $f^\sharp\circ\eta_X=f$, defined by
\[
f^\sharp(J) = \V_{x\in J} f(x).
\]
The $M$-equivariance of $f$ translates to left equivariance of $f^\sharp$ with respect to a set of generators of $\langle M\vert R_M\rangle$ and a set of generators of $\langle X\vert R_X\rangle$, and makes $f^\sharp$ a homomorphism of $\langle M\vert R_M\rangle$-modules as required. 
\end{proof}


\end{document}